\documentclass{elsart}

\usepackage{amssymb,latexsym}

\def\AA{{\mathcal A}}

\def\CC{{\mathcal C}}

\def\KK{{\mathcal K}}
\def\LL{{\mathcal L}}

\def\IN{{\mathbb{N}}}
\def\IQ{{\mathbb{Q}}}
\def\IR{{\mathbb{R}}}
\def\IS{{\mathbb{S}}}

\def\Low{\mathfrak{L}}

\def\TO{\Longrightarrow}
\def\In{\subseteq}

\def\into{\hookrightarrow}

\def\prefix{\sqsubseteq}

\def\mto{\rightrightarrows}

\def\id{{\rm id}}
\def\pr{{\rm pr}}
\def\inj{{\rm in}}

\def\dom{{\rm dom}}
\def\range{{\rm range}}
\def\graph{{\rm graph}}

\def\Inv{{\rm Inv}}

\def\Cantor{{\{0,1\}^\IN}}
\def\Baire{{\IN^\IN}}

\def\QED{$\hspace*{\fill}\Box$}

\def\ll#1{\ell_{#1}}

\newcommand{\SO}[1]{{{\bf\Sigma}^0_{#1}}}
\newcommand{\SI}[1]{{{\bf\Sigma}^1_{#1}}}

\newcommand{\dO}[1]{{{\Delta}^0_{#1}}}

\def\LPO{\text{\rm\sffamily LPO}}
\def\LLPO{\text{\rm\sffamily LLPO}}

\def\BFT{\mbox{\rm\sffamily BFT}}

\def\C{\mbox{\rm\sffamily C}}
\def\UC{\mbox{\rm\sffamily UC}}

\def\LPO{\mbox{\rm\sffamily LPO}}
\def\LLPO{\mbox{\rm\sffamily LLPO}}

\def\MLPO{\mbox{\rm\sffamily MLPO}}

\def\leqT{\mathop{\leq_{\mathrm{T}}}}
\def\nleqT{\mathop{\not\leq_{\mathrm{T}}}}

\def\leqW{\mathop{\leq_{\mathrm{W}}}}
\def\equivW{\mathop{\equiv_{\mathrm{W}}}}
\def\leqSW{\mathop{\leq_{\mathrm{sW}}}}
\def\equivSW{\mathop{\equiv_{\mathrm{sW}}}}

\def\nleqW{\mathop{\not\leq_{\mathrm{W}}}}

\def\lW{\mathop{<_{\mathrm{W}}}}
\def\lSW{\mathop{<_{\mathrm{sW}}}}

\def\bigtimes{\mathop{\mathsf{X}}}

\def\botW{\mathbf{0}}

\usepackage{theorem}
{\theorembodyfont{} \newtheorem{proposition}{Proposition}[section]}
{\theorembodyfont{} \newtheorem{theorem}[proposition]{Theorem}}
{\theorembodyfont{} \newtheorem{lemma}[proposition]{Lemma}}
{\theorembodyfont{\rm} \newtheorem{definition}[proposition]{Definition}}
{\theorembodyfont{} \newtheorem{corollary}[proposition]{Corollary}}
{\theorembodyfont{} \newtheorem{example}[proposition]{Example}}
{\theorembodyfont{} }
{\theorembodyfont{} }
{\theorembodyfont{} \newtheorem{conjecture}[proposition]{Conjecture}}
{\theorembodyfont{} }

\newenvironment{proof}{{\bf Proof.}\begin{normalsize}}{\end{normalsize}\QED\\}

\newenvironment{mycases}{\left\{\begin{array}{ll}}{\end{array}\right.}

\begin{document}

\begin{frontmatter}



\title{Closed Choice and a\\ Uniform Low Basis Theorem\thanksref{nrf}}
\thanks[nrf]{This work has been supported by the National Research Foundation of South Africa (NRF)
                  and the Japanese Society for Promotion of Sciences (JSPS)}


\author{Vasco Brattka}

\address{Laboratory of Foundational Aspects of Computer Science\\
             Department of Mathematics \& Applied Mathematics\\
             University of Cape Town, South Africa\\
            {\tt Vasco.Brattka@uct.ac.za}
}

\author{Matthew de Brecht}

\address{Graduate School of Informatics\\
             Kyoto University, Japan\\
   {\tt matthew@iip.ist.i.kyoto-u.ac.jp}}

\author{Arno Pauly}

\address{Computer Laboratory\\
             University of Cambridge, UK\\
             {\tt Arno.Pauly@cl.cam.ac.uk}}

\begin{abstract}
We study closed choice principles for different spaces.
Given information about what does not constitute a solution,
closed choice determines a solution.
We show that with closed choice one can characterize several
models of hypercomputation in a uniform framework using Weihrauch reducibility.
The classes of functions which are reducible to closed
choice of the singleton space, of the natural numbers, of Cantor space
and of Baire space correspond to the class of computable functions,
of functions computable with finitely many mind changes, of weakly computable
functions and of effectively Borel measurable functions, respectively.
We also prove that all these classes correspond to classes of non-deterministically
computable functions with the respective spaces as advice spaces.
The class of limit computable functions can be characterized with
parallelized choice on natural numbers.
On top of these results we provide further insights into algebraic
properties of closed choice. In particular, we prove that
closed choice on Euclidean space can be considered as ``locally compact choice''
and it is obtained as product of closed choice on the natural numbers and on Cantor space.
We also prove a Quotient Theorem for compact choice which shows that
single-valued functions can be ``divided'' by compact choice in a certain sense.
Another result is the Independent Choice Theorem, which provides a uniform proof that
many choice principles are closed under composition.
Finally, we also study the related class of low computable functions, which contains
the class of weakly computable functions as well as the class of functions computable with finitely many
mind changes. As one main result we prove a uniform version of the Low Basis Theorem
that states that closed choice on Cantor space (and the Euclidean space) is low computable.
We close with some related observations on the Turing jump operation and its initial topology.
\end{abstract}

\begin{keyword}
Computable analysis \sep Borel complexity \sep Weihrauch reducibility.
\end{keyword}
\end{frontmatter}

\section{Introduction}
\label{sec:introduction}

The basic task to be studied in the present paper is the following:

\begin{quote}
Given information about what does not constitute a solution, find a solution.
\end{quote}

The difficulty of this task depends strongly on the structure of the set of potential solutions.
In general, each represented space $(X, \delta)$ induces a topology, where a set $U\In X$
is open, if its characteristic function
\[\chi_U:X\to\IS,x\mapsto\left\{\begin{array}{ll}
  1 & \mbox{if $x\in U$}\\
  0 & \mbox{otherwise}
\end{array}\right.\]
is continuous with respect to the representation $\delta$ and a standard representation
$\delta_\IS$ of Sierpi\'nski space $\IS=\{0,1\}$ (which is equipped with the topology $\{\emptyset,\{1\},\{0,1\}\}$).
Such a standard representation of $\IS$ can be defined by
\[\delta_\IS(p)=1:\iff(\exists n)\;p(n)=0\]
for all $p\in\IN^\IN$.
Intuitively, the open sets are those for which membership can be continuously confirmed.
Each represented space then comes naturally with a representation $\delta^\circ$ of the
open sets, defined by
\[\delta^\circ(p)=U:\iff[\delta\to\delta_\IS](p)=\chi_U\]
for all $p\in\IN^\IN$.
Here $[\delta\to\delta_\IS]$ denotes the canonical function space representation (see \cite{Wei00})
of $\delta$ and $\delta_\IS$ (which is the exponential in the category of represented spaces).
The representation $\delta^\circ$ in turn induces a representation $\psi_-^X$ of the closed sets by
$\psi^X_-(p) = X \setminus \delta^\circ(p)$. The restriction to closed sets as solution sets arises
from the fact that they are exactly those sets for which one can continuously confirm membership in the complement.

We give some intuitive descriptions of equivalent versions of this very general representation for
concrete spaces that we will consider.

\begin{itemize}
\item $\mathbb{N}=\{0,1,2,...\}$, the set of natural numbers: the standard representation is
defined by $\delta_\IN(p):=p(0)$ and an equivalent way of defining $\psi_-^\IN$ is
by $\psi_-^\IN(p)=\{n\in\IN:n+1\not\in\range(p)\}$. That is $\psi_-^\IN(p)=A$, if $p$ is an
enumeration of all points that are not in $A$.
\item $\Cantor$, the Cantor space: the standard representation can be obtained by restricting the identity
on Baire space to Cantor space $\delta_\Cantor:=\id_\Baire|_\Cantor$.
In this case one can think that $\psi_-^\Cantor(p)=A$ if $p$ is a (potentially empty) enumeration of words $w_i\in\{0,1\}^*$ such that
$A=\Cantor\setminus\bigcup_{i=0}^\infty w_i\Cantor$. That is $p$ is a (potentially empty) enumeration of words $w_i$ such
that the corresponding balls exhaust the exterior of $A$.
\item $\Baire$, the Baire space: this case can be handled analogously to Cantor space, except that
the representation $\delta_\Baire$ is just the identity.
\item $\mathbb{R}$, the Euclidean real number line (and $\mathbb{R}^n$ in general):
for convenience we assume that we use some standard numbering $\overline{\ \ \rule[0mm]{0mm}{3mm}}:\IN\to\IQ$.
Then the Cauchy representation $\rho:\In\Baire\to\IR$ can be defined by
$\rho(p):=\lim_{n\to\infty}\overline{p(n)}$, where the domain $\dom(\rho)$ contains
only rapidly converging sequences, i.e.\ $p$ with $|\overline{p(i)}-\overline{p(j)}|<2^{-j}$ for all $i>j$.
Thus, a real number $x$ is represented by a rapidly converging sequence of rational numbers.
The representation $\psi_-^\IR$ can then be considered as follows: a name $p$ of a set $A$ is a sequence $(\langle a_i,b_i\rangle)_{i\in\IN}$
such that $A=\IR\setminus\bigcup_{i=0}^\infty (\overline{a_i},\overline{b_i})$. That is, intuitively, $p$ is a list of rational
intervals that exhaust the complement of $A$.
\item $\mathbb{I}:=[0,1]$, the real unit interval (and $\mathbb{I}^n$ in general): this can be treated by restricting the case of $\IR^n$.
\end{itemize}

For most spaces, closed choice is not computable. Thus, our interest lies on classifying the degree of incomputability,
that is the Weihrauch degree of closed choice, depending on the underlying space. Some of the arising
Weihrauch degrees are associated with certain models of type-2 hypercomputation, giving an independent
justification for our interest in closed choice. Additionally, as already demonstrated in
\cite{BG09b}, several important mathematical theorems share a Weihrauch degree with an appropriate
version of closed choice.

In recursion theory, a question closely related to our notion of closed choice has been studied.
Given a $\Pi_1^0$-class of Cantor space (which is a co-c.e.\ closed set in our terminology), what can we say about its elements?
It is known that a co-c.e.\ closed set may contain no computable points, but always contains a low point
\cite{JS72}. We present a stronger result, which takes the form that closed choice for Cantor space
is computable, if we replace the standard representation of the elements with another one, which just
renders the low points computable. On the side, we present a few results on the initial topology of the
Turing jump operator (called $\Pi$--topology by Joseph Miller, see \cite{Mil02a}).

\section{Weihrauch Reducibility}

This section serves to give a brief introduction into represented spaces, realizers, Weihrauch reducibility
and several associated operations. The basic reference for this section is \cite{Wei00}. While the study
of (variants of) Weihrauch-reducibility has commenced over a decade ago (\cite{Ste89}, \cite{Wei92a}, \cite{Wei92c}, \cite{Her96}),
the relevant sources for this section are \cite{BG09a}, \cite{BG09b} and \cite{Pau09}.

A significant ingredient of the theory of represented spaces is Baire space $\Baire$, i.e.\
the set of natural number sequences, equipped with the topology derived from the metric $d_\Baire$ which is
defined by $d_{\Baire}(u, u) = 0$ and $d_\Baire(u, v) = 2^{-\min \{n \mid u_n \neq v_n\}}$
for $u \neq v$. A useful property of Baire space to be exploited frequently is the existence of an effective and bijective
pairing  function $\langle \phantom{f}, \phantom{g}\rangle : \Baire \times \Baire \to \Baire$.
In the following we will denote partial functions using the symbol $\In$ as prefix and
multi-valued function using the double function arrow $\mto$. The term ``function'' or ``map'' might refer to any
of those but often we will indicate totality or single-valuedness, if relevant.

\begin{definition}[Representation]
A \emph{representation} $\delta$ of a set $X$ is a surjective single-valued (potentially partial)
function $\delta :\In\Baire\to X$. A \emph{represented space} $(X, \delta)$ is a set $X$ together
with a representation $\delta$ of it.
\end{definition}

Using represented spaces we can define the concept of a realizer. We denote the composition of
two (multi-valued) functions $f$ and $g$ either by $f\circ g$ or by $fg$.

\begin{definition}[Realizer]
Let $f : \In (X, \delta_X) \mto (Y, \delta_Y)$ be a multi-valued function between represented spaces.
A \emph{realizer} of $f$ is a single-valued function $F :\In \Baire \to \Baire$ satisfying
$\delta_Y \circ F(p)\in f \circ \delta_X(p)$ for all $p\in\dom(f\delta_X)$.
We use the notation $F \vdash f$ for expressing that $F$ is a realizer of $f$.
\end{definition}

As realizers are single-valued by definition, the statement that some function $F$ is a realizer
always implies its single-valuedness. Realizers allow us to transfer the notions of computability
and continuity and other notions available for Baire space to any represented space;
a function between represented spaces will be called {\em computable}, if it has a computable realizer, etc.
Now we have gathered the necessary provision to define Weihrauch reducibility ($\leqW$):

\begin{definition}[Weihrauch reducibility]
Let $f:\In X\mto Y$ and $g:\In U\mto V$ be multi-valued functions between represented spaces. Define $f \leqW g$, if there are computable
single-valued functions $K,H:\In\IN^\IN\to\IN^\IN$ satisfying $K \circ \langle \id, G \circ H \rangle \vdash f$ for all $G \vdash g$.
\end{definition}

We note that the relations $\leqW$ and $\vdash$ implicitly refer to the underlying representations, which
we will only mention explicitly if necessary.
The relation $\leqW$ is reflexive and transitive, thus it induces a partial order on the set of its equivalence classes
(which we refer to as {\em Weihrauch degrees}).
This partial order will be denoted by $\leqW$, as well. In this sense, $\leqW$ is a distributive bounded lattice
(for details see \cite{Pau09} and \cite{BG09a}).
We use $\equivW$ to denote equivalence regarding $\leqW$, $\lW$ for strict reducibility and $|_{\rm W}$ for incomparability.
There is a slightly stronger version of Weihrauch reducibility where the condition $K \circ \langle \id, G \circ H \rangle \vdash f$
is replaced by $K \circ G\circ H  \vdash f$. This {\em strong Weihrauch reducibility} is denoted by $f\leqSW g$.

We mention that the symbol $\leqW$ is also used to denote Wadge reducibility, which is in some sense a counterpart of
Weihrauch reducibility for sets and has been studied since the early 1970s, see \cite{Wad72,Wad83,Sel07}.
The double usage of $\leqW$ should not lead to confusion since Wadge reducibility is defined for sets and
Weihrauch reducibility for functions. We mention that some further information on the history of Weihrauch
reducibility is given in \cite{BG09a} and not repeated here.

We proceed to define a couple of useful operations.
While all definitions are given in terms of functions between represented spaces, they transfer directly to the according Weihrauch degrees.

The first operation is the coproduct, which plays the role of the supremum in the Weihrauch lattice.
By $X\coprod Y:=(\{0\}\times X)\cup(\{1\}\times Y)$ we denote the disjoint sum of two sets $X$ and $Y$
and if these spaces are represented spaces, then we assume that $X\coprod Y$ is equipped with the
canonical coproduct representation (see \cite{Pau09} for details).

\begin{definition}[Coproduct]
Let $f:\In X\mto Y$ and $g:\In W\mto Z$ be two multi-valued functions on represented spaces.
Then we define $f\coprod g:\In X\coprod W\mto Y\coprod Z$ by $(f\coprod g)(0,u):=\{0\}\times f(u)$ and $(f\coprod g)(1,u):=\{1\}\times g(u)$.
\end{definition}

One obtains that $H \vdash (f \coprod g)$ holds for exactly those $H$ satisfying $H(0w) = F(w)$ and $H(1w) = G(w)$
for some realizers $F \vdash f$ and $G \vdash g$ (that can depend on $w$).
We assume that the product $X\times Y$ of represented spaces $(X,\delta_X)$ and $(Y,\delta_Y)$
is represented with the canonical product representation $[\delta_X,\delta_Y]$ (see \cite{Wei00} for details).

\begin{definition}[Products]
Let $f:\In X\mto Y$ and $g:\In W\mto Z$ be two multi-valued functions on represented spaces.
Then we define $f\times g:\In X\times W\mto Y\times Z$ by $(f\times g)(x,w):=f(x)\times g(w)$.
\end{definition}

One obtains that $H \vdash (f \times g)$ holds for exactly those $H$ satisfying
$H(\langle u, v\rangle) = \langle F(u), G(v)\rangle$ for some realizers $F \vdash f$ and $G \vdash g$
(that might depend on $u,v$).

We say that a multi-valued map $f$ on represented spaces is {\em pointed}, if it contains at least
one computable point in its domain and we say that it is {\em idempotent}, if $f\times f\equivW f$.
In some cases the product and the coproduct are closely related.
If $f\times g$ is pointed and $f\coprod g$ is idempotent, then $f\coprod g\equivW f\times g$, since
\begin{eqnarray}
\mbox{$f\coprod g\leqW f\times g\leqW(f\coprod g)\times(f\coprod g)\leqW(f\coprod g)$},
\label{eqn:coproduct-product}
\end{eqnarray}
where pointedness of $f\times g$ is only required for the first reduction and idempotency of $f\coprod g$ only for the last one.
It is useful to consider a countable product of a multi-valued function with itself, which has been introduced in \cite{BG09a}.

\begin{definition}[Parallelization]
Let $f:\In X\mto Y$ be a multi-valued function on represented spaces.
We define the {\em parallelization} $\widehat{f}:\In X^\IN\mto Y^\IN$ by
$f(x_i)_{i\in\IN}:=\bigtimes_{i=0}^\infty f(x_i)$.
\end{definition}

We obtain that $H \vdash \hat{f}$ holds for exactly those $H$ satisfying
$H(\langle u_1, u_2, \ldots \rangle) = \langle F_1(u_1), F_2(u_2), \ldots\rangle$
for some realizers $F_i \vdash f$ for $i \in \mathbb{N}$ (that might depend on $u_i$).
We use the notation $\langle x_1, x_2, \ldots \rangle$ for the canonical countable pairing on Baire space.
In \cite{Pau09} a finite type of parallelization was introduced.
For any represented space $(X,\delta)$ we denote by $X^*=\bigcup_{i=0}^\infty(\{i\}\times X^i)$ the set
of all finite sequences over $X$ and we assume that $X^*$ is denoted by its canonical
standard representation $\delta^*$. For $f: \In X \mto Y$, we use $f^i$ to denote the $i$--fold product of $f$ with itself; and understand $f^0$ to be Weihrauch-equivalent to $\id_{\Baire}$.

\begin{definition}[Finite parallelization]
Let $f:\In X\mto Y$ be a multi-valued function on represented spaces.
We define the {\em finite parallelization} $f^*:\In X^*\to Y^*$ by
$f^*:=\coprod_{i=0}^\infty f^i$ with $f^*(i,x):=f^i(x)$ for all $(i,x)\in X^*$.
\end{definition}

Both types of parallelization form closure operators for the Weihrauch lattice,
which means $f \leqW \hat{f}$ and $\hat{f} \equivW\,\hat{\!\!\hat{f}}$, and $f \leqW g$ implies $\hat{f} \leqW \hat{g}$
and analogously for finite parallelization (see \cite{Pau09,Pau09b} and \cite{BG09a} for details).
It is easy to see that for pointed multi-valued functions idempotency is equivalent to $f\equivW f^*$.
It is interesting to mention that some variant of the (continuous) Weihrauch degrees has recently
be proved to be undecidable (see \cite{KSZ10}).

\section{Closed Choice}

Now we define the general version of closed choice for a represented space.

\begin{definition}[Closed Choice]
Let $(X,\delta)$ be a represented space. Then the {\em closed choice} operation
of this space is defined by
\[\C_X:\In\AA_-(X)\mto X,A\mapsto A\]
with $\dom(\C_X):=\{A\in\AA_-(X):A\not=\emptyset\}$.
\end{definition}

Here we assume that $\AA_-(X)$ is the set of closed subsets of $X$ equipped
with the negative information representation $\psi_-^X$ as defined in the
introduction. The computable points in $\AA_-(X)$ are called {\em co-c.e.\ closed sets}.
Intuitively, $\C_X$ takes as input a non-empty closed set in negative description (i.e.\ by some form
of enumeration of its complement) and it produces an arbitrary point of this set as output.
Hence, if we write $A\mapsto A$, then we mean that the multi-valued map $\C_X$ maps
the input $A$ (as a point in $\AA_-(X)$) to the set $A$ as a subset of $X$, namely the
set of possible function values.

Closed choice for particular spaces can characterize certain classes
of functions or degrees of mathematical theorems.
In \cite{GM09} it was proved that $\C_{\{0,1\}^\IN}$
is equivalent to the Hahn-Banach Theorem and to Weak K\H{o}nig's Lemma
and in \cite{BG09b} it was shown that $\C_\IN$ is equivalent to the Baire Category Theorem,
Banach's Inverse Mapping Theorem and several other theorems from functional analysis.
The following example shows that also many other classes that have been
considered can be characterized as classes of closed choice for certain spaces.

\begin{example}
We obtain $\C_{\{0\}}\equivW\C_\IS\equivW\id$, $\C_{\{0,1\}}\equivW\LLPO$
and, more generally, $\C_{\{0,1,...,n\}}\equivW\MLPO_{n+1}$.
\end{example}

Here $\MLPO_n$ and $\LLPO=\MLPO_2$ are taken from \cite{Wei92a}.
For $n\geq 1$ we consider $\MLPO_n:\In\IN^\IN\mto\IN$ as a multi-valued map with
\[\dom(\MLPO_n):=\{\langle p_1,...,p_n\rangle:(\exists i=1,...,n)\;p_i=\widehat{0}\}\]
and
\[\MLPO_n\langle p_1,...,p_n\rangle:=\{i:p_i=\widehat{0}\}.\]
Since $\LLPO$ is not idempotent (see \cite{BG09a}), it follows that closed choice
is not necessarily idempotent. However, it is a straightforward observation
that closed choice is always pointed, since $X$ is always a co-c.e.\ closed subset of itself.

\begin{lemma}[Pointedness]
If $X$ is a non-empty represented space, then $\C_X$ is pointed.
\end{lemma}

We get the following first result.

\begin{proposition}[Products]
\label{prop:product}
Let $X$ and $Y$ be non-empty represented spaces.
We obtain $\C_X\coprod\C_Y\leqW\C_X\times\C_Y\leqW\C_{X\times Y}$.
\end{proposition}
\begin{proof}
As mentioned in the introduction, coproducts are reducible to products for all pointed functions.
It is easy to prove that the Cartesian product
$P:\AA_-(X)\times\AA_-(Y)\to\AA_-(X\times Y),(A,B)\mapsto A\times B$
is computable and we obtain $\C_X\times\C_Y=\C_{X\times Y}\circ P$.
Hence $\C_X\times\C_Y\leqW\C_{X\times Y}$.
\end{proof}

We will see in Corollary~\ref{cor:coprod-product} that the inverse of the first reduction does not hold in general.
Also the second reduction cannot be reversed in general, as the following result shows.
We denote by $\widehat{n}:=nnn...\in\IN^\IN$ the constant sequence with value $n$.

\begin{proposition}[Products of choice for finite spaces]
\label{prop:finite}
Let $A$ and $B$ be finite sets, each with at least two elements and equipped
with the discrete representation and topology. Then
$\C_A\times\C_B\lW\C_{A\times B}$.
\end{proposition}
\begin{proof}
We assume that $A=\{0,...,n\}$ and $B=\{0,...,k\}$ with $n,k\geq1$
and we assume that $A$ is represented by $\delta_A(n\widehat{0}):=n$
with $\dom(\delta_A)=A\times\{\widehat{0}\}$.
Moreover, we assume $\psi_-^A(p)=\{i:i+1\not\in\range(p)\}$
with $\range(p)\In\{0,...,n+1\}$. This representation is computably
equivalent to the generic definition of $\psi_-^A$ given above.
Analogous assumptions are made for the representations $\delta_B$ and $\psi_-^B$
and $\delta_{A\times B}$ and $\psi_-^{A\times B}$.
We have $\C_A\times\C_B\leqW\C_{A\times B}$ by Proposition~\ref{prop:product}.

Let us now assume that $\C_{A\times B}\leqW\C_A\times\C_B$ holds.
Then there are computable functions $H,K$ such that
$F=H\langle\id,GK\rangle$ is a realizer of $\C_{A\times B}$ for any
realizer $G$ of $\C_A\times\C_B$. Now we consider $\widehat{0}=000...$
which represents $\psi_-^{A\times B}(\widehat{0})=A\times B$.
Then $(L,R):=[\psi_-^A,\psi_-^B]K(\widehat{0})$ is a pair of finite sets.\footnote{We are thankful to one of the referees
for providing a version of this paragraph that clarified and corrected the earlier version of it.}
For all $m$ and for all $p\in\dom(\psi_-^{A\times B})$ we have $\psi_-^{A\times B}(p)=\psi_-^{A\times B}(0^mp)$.
Moreover, by continuity of $K$ and since $A\times B$ is finite, there is $m\in\IN$ such that for all
$p\in\dom(\psi_-^{A\times B})$, we obtain that $(L',R')=[\psi_-^A,\psi_-^B]K(0^mp)$ implies $L'\In L$ and
$R'\In R$. By continuity of $H$ and since $A\times B$ is equipped with the discrete representation,
this $m$ can be taken such that $H\langle 0^mp,q\rangle$ is identical to $H\langle\widehat{0},q\rangle$ for
any fixed name $q$ of an element of $L'\times R'$. Finally, since there are only finitely many such $q$,
this $m$ can be selected as satisfying this property for all those $q$. Hence, for such $m$ we obtain
$F(0^mp)=H\langle\widehat{0},GK(0^mp)\rangle$. Since $\psi_-^{A\times B}$ is computably equivalent
to the representation $\psi_m$ given by $\psi_m(0^mp):=\psi_-^{A\times B}(p)$, we can assume without loss of generality
that there are computable functions $H,K$ such that $F=HGK$ is a realizer
of $\C_{A\times B}$ for any realizer $G$ of $\C_A\times\C_B$.


Let $M_j\subsetneqq M_{j-1}\subsetneqq...\subsetneqq M_0$ now be
a strictly decreasing sequence of non-empty subsets $M_i\In A\times B$.
Due to continuity of $K$ there is a monotone sequence of words $w_0\prefix w_1\prefix...\prefix w_j$
such that $\psi_-^{A\times B}(p_i)=M_i$ for $p_i:=w_i\widehat{0}$
and such that the sets
$(L_i,R_i):=[\psi_-^A,\psi_-^B]K(p_i)$ are component wise monotone as well.
That is $\emptyset\not=L_j\In L_{j-1}\In...\In L_0$
and $\emptyset\not=R_j\In R_{j-1}\In...\In R_0$. The cardinality of $A\times B$
is $(n+1)(k+1)$ and hence the longest strictly decreasing chain $(M_i)$ of non-empty sets
is one with length $j+1=(n+1)(k+1)$. The longest decreasing chain $(L_i,R_i)$
with the property that for each $i<j$ the left component or the right component
is strictly decreasing, i.e.\ $L_{i+1}\subsetneqq L_{i}$ or $R_{i+1}\subsetneqq R_{i}$,
has length $n+k+1$. For $n,k\geq1$ we have that
$n+k+1<(n+1)(k+1)$. Hence, there has to be at
least one $i<j$ such that $(L_i,R_i)=(L_{i+1},R_{i+1})$. By assumption there
is some element $x\in M_{i}\setminus M_{i+1}$.
For each element $y\in L_i\times R_i$ there is a realizer $G_y$ of $\C_A\times\C_B$
with $y=[\delta_A,\delta_B]G_yK(p_{i+1})$ and by assumption
$z:=[\delta_A,\delta_B]HG_yK(p_{i+1})\in M_{i+1}$ and hence $z\not=x$.
By continuity of $K$ there is an extension $w$ of $w_{i}$ such that
$\psi_-^{A\times B}(p)=\{x\}$ for $p:=w\widehat{0}$ and
$[\psi_-^A,\psi_-^B]K(p)\In(L_i,R_i)$ (where the inclusion is meant component wise).
Hence any realizer $G$ of  $\C_A\times\C_B$ selects an element $y=[\delta_A,\delta_B]GK(p)\in L_i\times R_i$
and thus $[\delta_A,\delta_B]HGK(p)\not=x$ in contrast to the fact that
$HGK$ is supposed to be a realizer of $\C_{A\times B}$.
Contradiction!
\end{proof}

Alternatively, one could prove this result by considering the level of the respective
operations, a concept that has been introduced by Hertling \cite{Her96}.
For one, one can prove directly $\MLPO_{n+1}\times\MLPO_{k+1}\leqW\LPO^{n+k}$, which
implies that $n+k$ is an upper bound on the level of $\MLPO_{n+1}\times\MLPO_{k+1}$.
On the other hand, $\LPO^{(n+1)(k+1)-1}$ can be reduced to any realizer of $\MLPO_{(n+1)(k+1)}$
(see Theorem~5.2.2 in \cite{Wei92c}), which implies that the level of $\MLPO_{(n+1)(k+1)}$ is at least $(n+1)(k+1)-1$.
Since Hertling proved that the level is preserved
downwards by Weihrauch reducibility, the desired result follows also from these observations.
We do not work out the details here.
For the simplest case of the set $\{0,1\}$ we get the following conclusion.

\begin{corollary}
$\C_{\{0,1\}}\times\C_{\{0,1\}}\lW\C_{\{0,1\}\times\{0,1\}}$.
\end{corollary}

We will see, however, that for many infinite spaces
we get a nicer behavior of products.
This is partially due to the following result.

\begin{proposition}[Surjections]
\label{prop:surjection}
Let $A$ and $B$ be represented spaces and let $s:\In A\to B$
be a computable surjection with a co-c.e.\ closed domain $\dom(s)$. Then $\C_B\leqW\C_A$.
\end{proposition}
\begin{proof}
If $s:\In A\to B$ is computable and $\dom(s)$ is co-c.e.\ closed in $A$,
then $S:\AA_-(B)\to\AA_-(A),M\mapsto s^{-1}(M)$
is computable too and if $s$ is surjective, then we obtain $\C_B=s\circ\C_A\circ S$, i.e.\ $\C_B\leqW\C_A$.
\end{proof}

As a consequence of this observation and Proposition~\ref{prop:product} we obtain the following sufficient criterion
for idempotency of choice.

\begin{corollary}[Idempotency]
\label{cor:idempotent}
Let $A$ be a represented space. If there is a computable surjection
$s:A\to A^2$, then $\C_A\times\C_A\equivW\C_{A\times A}\equivW\C_A$ and, in particular,
$\C_A$ is idempotent and hence also $\C_A^*\equivW\C_A$.
\end{corollary}

Since the spaces $\IN$, $\Cantor$, $\Baire$ and $\IN\times\Cantor$
admit computable and bijective pairing functions, we get the following conclusion.

\begin{corollary}
The choice principles $\C_\IN$, $\C_{\Cantor}$,
$\C_{\Baire}$ and $\C_{\IN\times\Cantor}$
are idempotent.
\end{corollary}

We close this section with the following example that shows that
in some cases choice commutes with parallelization and finite
parallelization and in other cases it does not.

\begin{example}
\label{ex:parallelization}
We obtain $\widehat{\C_{\{0,1\}}}\equivW\widehat{\LLPO}\equivW\C_{\Cantor}$, but $\widehat{\C_\IN}\equivW\lim\lW\C_{\Baire}$ and
$\C_\IN^*\equivW\C_\IN\equivW\C_{\IN^*}$, but
$\C_{\{0,1\}}^*\equivW\LLPO^*\lW\C_\IN\equivW\C_{\{0,1\}^*}$.
\end{example}

\section{Choice on Computable Metric Spaces}

In this section we want to study choice on certain large classes of computable
metric spaces. We recall that a {\em computable metric space} $(X,d,\alpha)$
is a separable metric space $(X,d)$ together with a numbering $\alpha:\IN\to X$
of a countable dense subset with respect to which the metric is computable.
By a {\em computable Polish space} we mean a computable metric space that is
also complete.
Usually, we will assume that computable metric spaces are represented
by their Cauchy representations $\delta_X$ (see \cite{Wei00}).
We use two different representation $\kappa_-$ and $\kappa$ to represent
the set $\KK(X)$ of compact subsets of a computable metric space $X$ (see \cite{BP03} for details).
Roughly speaking, a $\kappa_-$--name of a compact set $K\In X$ is a list of all finite covers
of $K$ by rational open balls, whereas a $\kappa$--name comes with the additional requirement
that all open balls in the cover actually have non-empty intersection with $K$.
That is, $\kappa_-$ provides negative information on the set $K$ (each cover allows
to exclude points) and $\kappa$ provides full information (each ball in the cover meets the set).
By $\KK_-(X)$ and $\KK(X)$ we denote the set of compact subsets represented by $\kappa_-$ and $\kappa$, respectively.
The compact sets that are computable with respect to $\kappa_-$ and $\kappa$
are called {\em co-c.e.\ compact} and {\em computably compact}, respectively.
We mention that a computable metric space is computably compact in itself if and
only if it is co-c.e.\ compact in itself.

Computable Polish spaces $X$ admit total computable and admissible representations $\delta:\Baire\to X$
(see, for instance, Corollary~4.4.12 in \cite{Bra99bx})
and computably compact computable metric spaces $X$ admit computable representations
$\delta:\Cantor\to X$ as we will prove next.
Two representations $\delta_1,\delta_2$ of the same set are said to be {\em (computably) reducible}
to each other, in symbols $\delta_1\leq\delta_2$, if there exists a computable function $F:\In\IN^\IN\to\IN^\IN$
such that $\delta_1=\delta_2\circ F$. Moreover, $\delta_1$ and $\delta_2$ are said to be
{\em (computably) equivalent}, in symbols $\delta_1\equiv\delta_2$, if $\delta_1\leq\delta_2$ and $\delta_2\leq\delta_1$ hold.
We recall that a representation of a computable metric space
is called {\em computably admissible}
if it is computably equivalent to the Cauchy representation of the space.

\begin{proposition}
\label{prop:compact-cantor}
Let $X$ be a computably compact computable metric space.
Then there is a surjective computable map $\varphi:\{0,1\}^\IN\to X$ that is also computably admissible.
\end{proposition}
\begin{proof}
Let $(X,d,\alpha)$ be a computably compact computable metric space.
We use a version $\delta_X:\In\{0,1\}^\IN\to X$ of the Cauchy representation, defined as follows
\[\delta_X(01^{n_0+1}01^{n_1+1}0...):=\lim_{i\to\infty}\alpha(n_i)\]
where $\dom(\delta_X)$ contains only those sequences of the given type which,
additionally, converge rapidly, i.e.\ such that $d(\alpha(n_i),\alpha(n_j))<2^{-j}$
for all $i\geq j$.
It is known that there exists a computably proper
and computably admissible representation $\delta:\In\{0,1\}^\IN\to X$
that is a restriction of $\delta_X$, see Corollary~4.6 in \cite{Wei03}.
Such a map is, in particular, computable and surjective and the fact that it
is computably proper implies that
$\delta^{-1}(K)$ is co-c.e.\ compact for any co-c.e.\ compact $K\In X$.
If $X$ itself is co-c.e.\ compact, then $A:=\dom(\delta)=\delta^{-1}(X)$
is also co-c.e.\ compact.
We claim that there is a total computable map $\iota:\{0,1\}^\IN\to\{0,1\}^\IN$
such that $A\In\range(\iota)\In\dom(\delta_X)$. A machine computing $\iota$ works as follows:
given an input $p\in\{0,1\}^\IN$
the machine checks in steps longer and longer prefixes $w$ of $p$
for the property
\begin{eqnarray}
\label{eqn:compact-cantor}
w\{0,1\}^\IN\In\{0,1\}^\IN\setminus A.
\end{eqnarray}
Since $A$ is co-c.e.\ closed, this property is c.e. in $w$. As long as the property cannot be verified, the machine
simultaneously checks whether the input is of the form $p=01^{n_0+1}01^{n_1+1}0...$
and whether the property $d(\alpha(n_i),\alpha(n_j))<2^{-j}$
is satisfied for all $i\geq j$ such that $01^{n_i+1}0$ is completely included in $w$.
If the latter property is positively verified, then the output is extended such that
it matches $p$ up to the corresponding $01^{n_i+1}0$. If, at any time, property (\ref{eqn:compact-cantor})
is positively verified, then it is clear that $p\not\in A$ and the processing of the input is
stopped and the output is extended just by infinitely many repetitions of the
last block $1^{n_i+1}0$ (if no block has been written at this stage, then an arbitrary block $01^{n_0+1}$ is repeated
infinitely often as output).
If the input $p$ is not of the form $p=01^{n_0+1}01^{n_1+1}0...$, then the test for property (\ref{eqn:compact-cantor})
will eventually be positive.
It is clear that altogether this machine computes a function $\iota:\{0,1\}^\IN\to\{0,1\}^\IN$
such that $A\In\range(\iota)\In\dom(\delta_X)$. This guarantees that $\varphi:=\delta_X\circ\iota$
is computable, total and surjective.
Since $\dom(\delta)=A\In\range(\iota)$ it follows that $\delta=\varphi\circ\iota^{-1}$.
Since $\iota$ is computable, also $\iota^{-1}$ is computable (see Corollary~\ref{cor:compact-inversion})
and hence it follows that $\varphi$ is computably admissible.
\end{proof}

Hence we obtain the following corollary. The first statement is a consequence of Proposition~\ref{prop:surjection}
and the second statement a consequence of the previous Proposition~\ref{prop:compact-cantor}.

\begin{corollary}
\label{cor:polish-compact}
Let $X$ be a computable Polish space. Then $\C_X\leqW\C_{\Baire}$.
If, additionally, $X$ is computably compact, then $\C_X\leqW\C_{\Cantor}$.
\end{corollary}

We say that $\iota:A\to B$ is a {\em computable embedding}, if $\iota$
is computable and injective and its partial inverse $\iota^{-1}$ is computable too.
Now we can use the Embedding Theorem~3.7 from \cite{BG09} in order to
obtain the following proposition.

\begin{proposition}
\label{prop:embedding}
Let $A$ and $B$ be computable metric spaces and let $\iota:A\to B$ be
a computable embedding such that $\range(\iota)$ is co-c.e.\ closed in $B$.
Then $\C_A\leqW\C_B$.
\end{proposition}
\begin{proof}
From Theorem~3.7 in \cite{BG09} it follows that for a computable embedding
$\iota:A\to B$ with co-c.e.\ closed range $\iota(A)$ the map
$J:\AA_-(A)\to\AA_-(B),M\mapsto\iota(M)$ is computable.
We obtain $\C_A=\iota^{-1}\circ\C_B\circ J$ and hence $\C_A\leqW\C_B$.
\end{proof}

We recall that a metric space is called perfect, if it has no isolated points.
In Proposition~6.2 in \cite{BG09} it has been proved that any non-empty
perfect computable Polish space is rich, i.e.\ admits a computable
embedding $\iota:\{0,1\}^\IN\to X$ and in this case $\range(\iota)$ is
automatically co-c.e.\ closed. Hence we obtain the following corollary.

\begin{corollary}
Let $X$ be a computable Polish space. If $X$ is rich and, in particular, if
$X$ is non-empty and has no isolated points, then $\C_{\Cantor}\leqW\C_X$.
\end{corollary}

Together with Corollary~\ref{cor:polish-compact} we get the following
corollary (which has essentially been proved in \cite{GM09} already).

\begin{corollary}
Let $X$ be a computably compact metric space, which is non-empty and has no isolated points,
then $\C_\Cantor\equivW\C_X$.
\end{corollary}

Thus, $\C_\Cantor$ can be identified with ``compact choice'' for a
very large class of compact spaces. In particular, we obtain the following corollary.

\begin{corollary}
$\C_\Cantor\equivW\C_{[0,1]}\equivW\C_{[0,1]^\IN}$.
\end{corollary}

We would like to show that $\C_\IN\times\C_\Cantor$ plays a similar role for locally compact spaces
as $\C_\Cantor$ does for compact spaces.
The following lemma plays a role in the proof of the next result
and it is worth being formulated separately.

\begin{lemma}
\label{lem:extension}
Let $K$ be a non-empty computably compact computable metric space.
Then $\C_K:\In\AA_-(K)\mto K$ has a total extension $\C_K':\AA_-(K)\mto K$ with $\C_K\equivW\C_K'$.
\end{lemma}
\begin{proof}
The set $\{A\in\AA_-(K):A=\emptyset\}$ is c.e.\ open for co-c.e.\ compact $K$.
Since $K$ is computably compact, we can assume by Proposition~\ref{prop:compact-cantor}
without loss of generality that
$K$ is represented by a total representation $\delta:\{0,1\}^\IN\to K$.
Hence $\C_K$ can be extended to a suitable $\C_K'$ as follows: a realizer $F$ of $\C_K$
is modified to a map $G$ such that never anything else but zeros and ones are written on the output tape
and as soon as the empty set is detected as input, the output is just continued
with constant zeros. In any other respect, the map $G$ behaves exactly as $F$.
Due to totality of $\delta$, this output of $G$ is in the domain of $\delta$.
The modification guarantees that the empty set as input leads to some infinite output
and non-empty sets are treated by $G$ exactly as by $F$.
The construction shows that $\C_K'$ is reducible to $\C_K$.
The reverse direction follows since $\C_K'$ is an extension of $\C_K$.
\end{proof}

We note that not every multi-valued operation has a total equivalent extension
(as robust division shows, see \cite{Pau09b}).

Classically, a space $X$ is called {\em $\sigma$--compact} or $K_\sigma$--space,
if it can be written as a countable union of compact sets. For many spaces
this property is somewhat weaker than local compactness,
this holds in particular for represented Hausdorff spaces.
The induced topology of every represented space is known to be 	hereditarily Lindel\"of (see Lemma~2.5 in \cite{Bos08f})
and this means that if it is, additionally, a Hausdorff space,
then local compactness implies $\sigma$--compactness.
This is the reason why we speak about ``locally compact choice'' for short.
We say that $X$ is a {\em computable $K_\sigma$--space}, if $X$ is a computable metric
space, such that there exists a computable sequence $(K_i)_{i\in\IN}$ of non-empty
computably compact sets with $X=\bigcup_{i=0}^\infty K_i$.

\begin{proposition}[Locally compact choice]
\label{prop:locally-compact-choice}
Let $X$ be a computable $K_\sigma$--space.
Then $\C_X\leqW\C_\IN\times\C_{\Cantor}$.
\end{proposition}
\begin{proof}
We consider the total extensions $\C_{K_i}'$ of choice that exist
according to a uniform version of Lemma~\ref{lem:extension}.
Using a uniform version of Corollary~\ref{cor:polish-compact}, we obtain
\[F:=\C_{K_0}'\times\C_{K_1}'\times\C_{K_2}'\times...\leqW\widehat{\C_{\Cantor}}\equivW\C_{\Cantor}.\]
Given a closed set $A\In X$ we can compute the sequence
$(A\cap K_n)_{n\in\IN}$ of co-c.e.\ compact sets and hence we can
enumerate the set $\{n\in\IN:A\cap K_n=\emptyset\}$.
This implies that we can find an $n$ such that $A\cap K_n\not=\emptyset$
with the help of $\C_\IN$.
Moreover, $F((A\cap K_n)_{n\in\IN})$ can be obtained with the
help of $\C_{\Cantor}$, as indicated above.
Altogether, this shows $\C_X\leqW\C_\IN\times\C_{\Cantor}$.
\end{proof}

This result can even be generalized to the case that the $K_\sigma$--space is only co-c.e.\ compact
in the sense that the sequence $(K_i)_{i\in\IN}$ is only a computable sequence of co-c.e.\ compact sets.
However, in this case the uniform version of Lemma~\ref{lem:extension} needs some extra attention
since the extensions $\C_K'$ might not always produce a value in $K$ (but only some infinite sequence).
By Proposition~\ref{prop:product} we have $\C_\IN\times\C_{\Cantor}\leqW\C_{\IN\times\Cantor}$.
On the other hand, we can apply the previous proposition to the
$K_\sigma$--space $\IN\times\Cantor$ (with $K_n:=\{n\}\times\Cantor$)
and we get the inverse reduction. We can also apply the previous
proposition to $\IR^k$ (with $K_n:=[-n,n]^k$).

\begin{corollary}
\label{cor:real-choice}
$\C_{\IR^k}\equivW\C_\IR\equivW\C_{\IN\times\Cantor}\equivW\C_\IN\times\C_{\Cantor}$
for all $k\geq1$.
\end{corollary}

We mention that by the Theorem of Hurewicz (see Theorem~7.10 in \cite{Kec95}) any Polish
space which is not $K_\sigma$ admits an embedding $\iota:\Baire\to X$ such that
$\range(\iota)$ is closed. Using relativized topological versions of Propositions~\ref{prop:embedding}
and \ref{prop:locally-compact-choice} and Corollary~\ref{cor:polish-compact}
we obtain the following dichotomy.

\begin{corollary}[Dichotomy]
\label{cor:dichotomy}
If $X$ is a Polish space, then there is an oracle such that either $\C_X\leqW\C_\IR$ or $\C_\Baire\equivW\C_X$,
relatively to that oracle (i.e.\ with continuous reductions).
\end{corollary}

In other words, topologically the interval between $\C_\IR$ and $\C_\Baire$ is not inhabited
by choice principles of Polish spaces.
It is not too hard to see that for many computable metric spaces $X$ that are not $K_\sigma$,
such as $\IR^\IN$, $\CC[0,1]$ and $\ll{p}$,
there is a computable embedding $\iota:\Baire\to X$ with a co-c.e.\ closed image. Hence we get the following
corollary of Proposition~\ref{prop:embedding}.

\begin{corollary}
$\C_\Baire\equivW\C_{\IR^\IN}\equivW\C_{\ll{p}}\equivW\C_{\CC[0,1]}$ for all computable real $p\geq1$.
\end{corollary}

The results mentioned so far in this section are mostly applicable to
Polish spaces. We mention two further examples for non-Polish spaces.
Any sequence $(x_n)_{n\in\IN}$ in $X$ can be seen as a surjection from $\IN$
onto the range of the sequence. Hence we obtain the following corollary.

\begin{corollary}
Let $X$ be a represented space and let $(x_n)_{n\in\IN}$ be a computable
sequence in $X$ with $R:=\{x_n:n\in\IN\}$. Then $\C_R\leqW\C_\IN$.
\end{corollary}

This can, in particular, be applied to the rational numbers as a subspace
of Euclidean space.

\begin{corollary}
$\C_\IQ\equivW\C_\IN$, independently of whether $\IQ$ is equipped
with the discrete representation and topology or with the Euclidean one.
\end{corollary}

The irrational numbers are computably homeomorphic to Baire space
(with respect to the Euclidean topology and via their continued fraction representation)
and hence we get the following conclusion.

\begin{corollary} $\C_{\IR\setminus\IQ}\equivW\C_{\Baire}$.
\end{corollary}

\section{Compact Choice, Quotients and Join-Irreducibility}

The following theorem shows that any single-valued function $f$ that can
be computed from compact choice and another function $g$ can already
be computed from $g$ alone. Thus, we can ``divide'' by compact choice
in such a situation. This result generalizes Corollary~8.8 in \cite{BG09a}.

\begin{theorem}[Quotients]
\label{thm:quotient}
Let $X$ be a represented space and $Y$ be a computable metric space
and let $g$ be a multi-valued function on represented spaces.
If $f:\In X\to Y$ is single-valued and $f\leqW\C_{\Cantor}\times g$, then
$f\leqW g$.
\end{theorem}
\begin{proof}
We use the Cauchy representation $\delta_Y$ for $Y$ and canonical projections
$\pi_i:\Baire\to\Baire$ with $\pi_1\langle p,q\rangle=p$ and $\pi_2\langle p,q\rangle=q$.
Now let $f:\In X\to Y$ be such that $f\leqW\C_{\Cantor}\times g$.
Hence there are computable functions $H$ and $K$ such that
$K\langle\id,PH\rangle$ is a realizer of $f$ for any
realizer $P$ of $\C_{\Cantor}\times g$.
Since $H$ and $K$ are computable, as well as the Cartesian product on compact sets,
it follows from Theorem~3.3 in \cite{Wei03}
that there is a computable function $S:\In\Baire\to\Baire$ with
\[\kappa_-S\langle p,q\rangle=\delta_YK\langle\{p\}\times\langle\kappa_-\pi_1H(p)\times\{q\}\rangle\rangle\]
for all $p\in\dom(f\delta_X)$ and suitable $q$.
We now consider the function $T:\In\Baire\to\Baire$ with
$T(p)=S\langle p,G\pi_2H(p)\rangle$. Whenever $G$ is a realizer of $g$, then $T$ is a realizer
of the function $F:\In X\to\KK_-(Y),x\mapsto\{f(x)\}$.
Hence, $F\leqW g$. If the space $Y$ is a computable metric space, then
$\inj:Y\to\KK_-(Y),x\mapsto\{x\}$ has a computable inverse (see Lemma 6.4 in \cite{Bra08})
and it follows that $f=\inj^{-1}\circ F$. That implies $f\leqW F$.
\end{proof}

We note that this theorem can be generalized to larger classes of spaces $Y$.
The only property that is exploited is that the injection $\inj:Y\to\KK_-(Y)$ has
a computable inverse. We obtain some straightforward corollaries.

\begin{corollary}
\label{cor:quotient-cantor}
Let $X$ be a represented space and $Y$ be a computable metric space.
If $f:\In X\to Y$ is single-valued and $f\leqW\C_{\Cantor}$, then
$f$ is computable.
\end{corollary}

This is just Corollary~8.8 from \cite{BG09a}.
Together with Corollary~\ref{cor:real-choice} we obtain the following
result, which is new.

\begin{corollary}
\label{cor:quotient-reals}
Let $X$ be a represented space and $Y$ be a computable metric space.
If $f:\In X\to Y$ is single-valued and $f\leqW\C_{\IR}$, then
$f\leqW\C_\IN$.
\end{corollary}

By exploiting the distributivity of the Weihrauch lattice discovered in
\cite{Pau09}, a restricted version of
Theorem~\ref{thm:quotient} could be obtained, using coproducts instead of products.
Combined with the observation that coproducts are the suprema in the Weihrauch lattice,
and the usefulness of the decomposition into products presented in Corollary~\ref{cor:real-choice},
it seems sensible to explore whether any of our principles of closed choice can be expressed as a
supremum of other degrees. The negative answer is a consequence of the next result.
To formulate it, we define the concept of join-irreducibility in the Weihrauch lattice.

\begin{definition}[Join-irreducibility]
A multi-valued function $f$ on represented spaces is called {\em join-irreducible},
if $f \equivW\coprod_{n \in \mathbb{N}} f_n$ implies
the existence of an $n_0 \in \mathbb{N}$ with $f \equivW f_{n_0}$.
\end{definition}

We note that for finitely many $f_n$, this is exactly the ordinary lattice theoretic
concept of join-irreducibility. For countably many $f_n$, this concept might be called
$\sigma$--join-irreducibility (see \cite{Sel07}). However, this is also not quite
appropriate since the coproduct $\coprod_{n\in\IN}f_n$ is not
necessarily the supremum of the $f_n$. This is correct for continuous reducibility,
but not for the computable case. We refrain to introduce another name and
call the above concept just join-irreducibility, which is justified since we will
basically only apply it in a situation with finitely many $f_n$.

If $f:\In X\mto Y$ is a function between represented spaces, with representation
$\delta$ of $X$, then we define
$f_A$ for each set $A\In\Baire$ as follows. We let $(X_A,\delta|_A)$ be
the represented space with $X_A:=\delta(A)$ and the restriction $\delta|_A$
of $\delta$ to $A$. Then $f_A:\In X_A\mto Y$ is the restriction of $f$ to
the represented space $(X_A,\delta_A)$.
That is, we obtain $F|_A \vdash f_A$ if $F \vdash f$.
Using this concept, we get the following sufficient criterion for join-irreducibility.

\begin{lemma}[Join-irreducibility]
\label{lem:join-irreducible}
Let $(X,\delta_X)$ and $Y$ be represented spaces.
Assume that for some multi-valued function $f:\In X\mto Y$
the equivalence $f\equivW f_A$ holds for each
non-empty set $A\In\Baire$ that is clopen in $\dom(f\delta_X)$. Then $f$ is join-irreducible.
\end{lemma}
\begin{proof}
Assume $f \leqW \coprod_{n \in \mathbb{N}} f_n$.
Then there exists a computable function $\mathcal{N} :\In\Baire \to \mathbb{N}$ with
$f_{\mathcal{N}^{-1}(n)} \leqW f_n$ and $\dom(\mathcal{N})=\dom(f\delta_X)$.
There has to be an $n_0 \in \mathbb{N}$,
so that $\mathcal{N}^{-1}(n_0) \neq \emptyset$, and due to continuity of $\mathcal{N}$,
this set is closed and open in $\dom(f\delta_X)$. Thus, by the assumption, we have
$f \leqW f_{\mathcal{N}^{-1}(n_0)} \leqW f_{n_0}$. The other direction is trivial.
\end{proof}

If we take away finitely many small open rational balls from $\IN$, $\Cantor$, $\Baire$
or $\IR$, respectively, such that the remainder is non-empty, then the remainder is still
large enough to simulate closed choice of the entire space within this subspace.
This is why closed choice for all these spaces satisfies the above criterion for
join-irreducibility.

\begin{corollary}
$\C_\mathbb{N}$, $\C_\Cantor$, $\C_\Baire$ and $\C_\mathbb{R}$ are join-irreducible.
\end{corollary}

Another consequence is that the coproduct (i.e.\ the supremum) of $\C_\IN$ and $\C_\Cantor$
is strictly below the product.

\begin{corollary}
\label{cor:coprod-product}
$\C_\mathbb{N} \coprod \C_\Cantor \lW \C_\mathbb{N} \times \C_\Cantor$.
\end{corollary}

This corollary also shows that the coproduct of two idempotent functions is not
necessarily idempotent (see Equation~(\ref{eqn:coproduct-product})).

\begin{corollary}
$\C_\mathbb{N} \coprod \C_\Cantor$ is not idempotent.
\end{corollary}

\section{Unique Choice and Inversion}

In this section we briefly discuss a variant of choice,
which we call unique choice. This is choice restricted
to the special case of singletons. We only formulate
unique choice for Hausdorff spaces in order to guarantee
that singletons are closed.

\begin{definition}[Unique Closed Choice]
Let $(X,\delta)$ be a represented Hausdorff space.
We consider the injection $\inj_X:X\into\AA_-(X),x\mapsto\{x\}$.
The partial inverse $\UC_X:\In\AA_-(X)\to X$ of this injection is called
{\em unique closed choice} operation of the space $X$.
\end{definition}

Since unique choice $\UC_X$ is a restriction of choice $\C_X$,
it is clear that $\UC_X\leqW\C_X$ holds.
In some cases we can say more.
In case of $\IN$ it turns out that unique choice is not easier than full choice.
The proof idea is very similar to the proof idea of Proposition~3.3 in \cite{BG09b},
where $\C_\IN$ is reduced to finite choice. We only describe it informally here.

\begin{proposition}
$\UC_\IN\equivW\C_\IN$.
\end{proposition}
\begin{proof}
It is clear that $\UC_\IN\leqW\C_\IN$.
We prove $\C_\IN\leqW\UC_\IN$ by an intuitive description
of a suitable algorithm. Given an enumeration $n_0,n_1,...$
of the complement of a set $A\In\IN$, we choose $c=0$ as starting
candidate for a potential element in $A$ and we choose $j=0$ as starting
position to keep track of where we have to change our mind.
In steps $i=0,1,...$ we inspect the enumeration $n_i$ in order to find the
candidate $c$ and simultaneously we start to generate as output
a negative description of $\{j\}$ by enumerating all numbers $k>j$.
Whenever some $i$ with $c=n_i$ is found, we choose
as new candidate $c$ the minimal element $c=\min(\IN\setminus\{n_0,...,n_i\})$.
Whenever that happens, we choose $j=\max\{i,m+1\}$ as new position, where $m$ is
the largest number that has been produced on the output and now we start
to produce as output a negative description of $\{j\}$ by enumerating all numbers
$m+1,...,j-1$ (if there are any) and then all numbers $k>j$,
while we continue to inspect the sequence $n_{i+1},n_{i+2},...$ to find the new candidate $c$.
If we continue like this, then eventually we will find a candidate $c$ that is actually in $A$
and hence not in the enumeration of the $n_i$. The output will then be a negative description
of $\{j\}$ for some number $j$ that is larger than or equal to the last position in the enumeration
where we had to change our candidate. That is, the number $j$ together with the original
enumeration $n_0,n_1,...$ allows to identify the candidate $c$. The number $j$ can
be obtained from the output with the help of unique choice $\UC_\IN$.
\end{proof}

In case of Baire space we formulate the following conjecture.

\begin{conjecture}
$\UC_\Baire\equivW\C_\Baire$.
\end{conjecture}

On the other hand, unique choice $\UC_\Baire$ is also not too simple.
One can easily see that $\lim\leqW\UC_\Baire$ holds for the limit map
\[\lim:\In\Baire\to\Baire,\langle p_0,p_1,p_2,...\rangle\mapsto\lim_{i\to\infty}p_i\]
and with the methods of the next section it also follows that the cone below $\UC_\Baire$ is closed under composition.
Hence, $\UC_\Baire$ cannot be located on any finite level of the Borel hierarchy.
This can also be deduced from the fact that there are co-c.e.\ closed singletons $\{p\}\In\IN^\IN$
such that $p$ is hyperarithmetical, but not arithmetical (see Propositions~1.8.62 and 1.8.70 in \cite{Nie09}).
We obtain the following corollary as a direct consequence of
Corollaries~\ref{cor:quotient-cantor} and \ref{cor:quotient-reals}
and the previous proposition and the observation that $\UC_\IN\leqW\UC_\IR$ holds.

\begin{corollary}
\label{cor:unique-choice}
$\UC_{\Cantor}\equivW\C_{\{0\}}\equivW\id$ and $\UC_\IR\equivW\C_\IN$.
\end{corollary}

We will use the inversion and the graph map as follows
\begin{itemize}
\itemsep 0.3cm
\item $\Inv_{X,Y}:\In\CC(X,Y)\times Y\to X,(f,y)\mapsto f^{-1}(y)$, where\\
        $\dom(\Inv_{X,Y}):=\{(f,y):f$ injective and $y\in\dom(f^{-1})\}$
\item $\graph_{X,Y}:\CC(X,Y)\to\AA_-(X\times Y),f\mapsto\graph(f)$, where\\
        $\graph(f):=\{(x,y)\in X\times Y:f(x)=y\}$.
\end{itemize}

For computable metric spaces $X$ and $Y$ the map $\graph_{X,Y}$ is known
to be computable (see \cite{Bra08}).
It turns out that the map $\Inv_{X,Y}$ is reducible to unique choice of $X$.

\begin{theorem}[Inversion operator]
Let $X$ and $Y$ be computable metric spaces.
Then $\Inv_{X,Y}\leqW\UC_X$.
\end{theorem}
\begin{proof}
In \cite{Bra08} we have established the formula
\[f^{-1}(y)=\inj_X^{-1}\circ\sec(\graph_{X,Y}(f),y),\]
where
$\sec:\AA_-(X\times Y)\times Y\to X,(A,y)\mapsto A_y:=\{x\in X:(x,y)\in A\}$
is the computable section map (see \cite{Bra08}).
Altogether, this shows $\Inv_{X,Y}\leqW\inj_X^{-1}=\UC_X$.
\end{proof}

As a corollary we get that in particular any specific inverse of a computable map is reducible to unique
choice. We can generalize this non-uniform result even to the case of non-injective maps and ordinary choice.
We note that the inverse $f^{-1}:\In Y\mto X,y\mapsto f^{-1}\{y\}$ exists as a multi-valued map
for any single-valued $f:X\to Y$.

\begin{theorem}[Inversion]
\label{thm:inversion}
Let $X$ and $Y$ be computable metric spaces. If\linebreak
$f:X\to Y$ is computable, then $f^{-1}\leqW\C_X$
and if $f$ is also injective, then $f^{-1}\leqW\UC_X$.
\end{theorem}
\begin{proof}
If $f:X\to Y$ is computable, then $F:\AA_-(Y)\to\AA_-(X),A\mapsto f^{-1}(A)$
is computable too and we obtain
\[f^{-1}(y)=\C_X\circ F\circ\inj_Y(y),\]
i.e.\ $f^{-1}\leq\C_X$
and if $f$ is also injective, then we obtain
$f^{-1}(y)=\inj_X^{-1}\circ F\circ\inj_Y(y)$,
i.e.\ $f^{-1}\leqW\UC_X$.
\end{proof}

We mention that a multi-valued function $f$ on represented spaces
is called {\em weakly computable},
if $f\leqW\C_{\{0,1\}^\IN}$ and $f$ is called {\em computable with finitely many mind changes}
if it can be computed on a Turing machine that revises its output at most finitely many times
for each particular input. In Theorem~\ref{thm:choice-natural} we will show that the latter
is equivalent to $f\leqW\C_\IN$.
We get the following result as a corollary of Theorem~\ref{thm:inversion}
and Corollary~\ref{cor:unique-choice}.

\begin{corollary}[Compact inversion]
\label{cor:compact-inversion}
Let $X$ and $Y$ be computable metric spaces and let $X$ be computably compact.
If $f:X\to Y$ is computable, then $f^{-1}$ is weakly computable, if $f$ is also injective,
then $f^{-1}$ is even computable.
\end{corollary}

The second part of the statement was known as such (see, for instance, \cite{Bra08}).
The following corollary is also a consequence of Theorem~\ref{thm:inversion}
and Corollary~\ref{cor:unique-choice}.

\begin{corollary}[Locally compact inversion]
Let $X$ be a computable $K_\sigma$--space and let $Y$ be a computable metric space.
If $f:X\to Y$ is computable, then $f^{-1}\leqW\C_\IR$ and if $f$ is also injective then
$f^{-1}\leqW\C_\IN$, hence $f^{-1}$ is computable with finitely many mind changes.
\end{corollary}

These results are not necessarily optimal. For instance, it is known that the
inverse of an injective computable map $f:\IR\to\IR$ is even computable.
However, for this result one has to exploit additional properties of $\IR$,
such as connectedness properties (see \cite{Bra08}).
We give some example that shows that the inversion results do
not hold true for arbitrary represented spaces.
By $\IR_<$ we denote the set of real numbers equipped with the left cut representation $\rho_<$,
which represents a real number $x$ by an enumeration of all rational numbers $q<x$ (see \cite{Wei00}).

\begin{example}
Let $\IR$ denote the real number represented with the Cauchy representation
and let $\IR_<$ denote the real number denoted with the left cut representation.
The identity $f:\IR\to\IR_<,x\mapsto x$ is computable and its inverse
$f^{-1}:\IR_<\to\IR$ is known to be equivalent to $\lim$ (see Proposition~3.7 in \cite{BG09b} and Exercise~8.2.12 in \cite{Wei00}).
In particular, $f^{-1}$ is not reducible to $\C_\IR$.
\end{example}

\section{Choice on Baire Space and Non-Deterministic Computability}
\label{sec:non-determinism}

In this section we will compare the power of choice for certain spaces with
models of hypercomputation that have been considered.
This approach to classify models of hypercomputation in terms of Weihrauch reducibility
has been started in \cite{Pau09c}.
Here, the relevant models of hypercomputation are non-determi\-nistically computable
functions and functions computable with revising computations in the sense of Martin Ziegler \cite{Zie07a,Zie07}.
The latter ones are also known as functions computable with finitely many mind-changes,
for instance in learning theory \cite{BY09,BY09a}.

In \cite{Zie07a} Martin Ziegler has introduced a concept of non-deterministically
computable functions. We generalize this concept to advice spaces that are subsets of Baire space
and we prove that this concept can be characterized by choice for the advice space.
This characterization yields some interesting consequences.

\begin{definition}[Non-deterministic computability]
\label{def:non-deterministic}
Let $(X,\delta_X)$, $(Y,\delta_Y)$ be represented spaces and let $A\In\Baire$.
A function $f:\In X\mto Y$ is said to be {\em non-deterministically computable with advice space $A$},
if there exist two computable functions
$F_1,F_2:\In \Baire\to\Baire$ such that $\langle\dom(f\delta_X)\times A\rangle\In\dom(F_2)$ and
for each $p\in\dom(f\delta_X)$ the following hold:
\begin{enumerate}
\item $(\exists r\in A)\;\delta_\IS F_2\langle p,r\rangle=0$,
\item $(\forall r\in A)(\delta_\IS F_2\langle p,r\rangle=0\TO\delta_YF_1\langle p,r\rangle\in f\delta_X(p))$.
\end{enumerate}
\end{definition}

Here $A\In\Baire$ is considered as subspace of Baire space.
Intuitively, the set $A$ serves as a set of possible advices that can give extra
support to the computation. Any computation can be successful
or it can fail, which is indicated by the output of $F_2$
(where ``$1$'' means the advice is recognized to fail after finite time
and ``$0$'' means the advice is successful in the long run).
That is $F_2$ can be considered as a realizer of a function $f_2:\In\IN^\IN\to\IS$.
The set
\[A_p:=\{r\in A:\delta_\IS F_2\langle p,r\rangle=0\}=\{r\in A:f_2\langle p,r\rangle=0\}\]
is the set of successful advices for input $p\in\dom(f\delta_X)$.
Intuitively, $F_2$ is a method to recognize unsuccessful advices and
$F_1$ is a method to determine the output of the computation for successful advices.
The two conditions then express intuitively that for each fixed admissible input the following hold:
\begin{enumerate}
\item There exists a successful advice for this input.
\item Each successful advice produces a correct output.
\end{enumerate}
Functions that are non-deterministically computable in the sense of \cite{Zie07a}
are non-deterministically computable with full Baire space $\Baire$ as advice space\footnote{The
advice space is not made explicit in \cite{Zie07a}, but we conclude implicitly that the advice space $\Baire$ is meant.}.
Now we can prove the following equivalence.

\begin{theorem}[Non-deterministic computability]
\label{thm:non-deterministic}
Let $X$ and $Y$ be represented spaces, $A\In\Baire$ and
let $f:\In X\mto Y$ be a multi-valued function.
Then the following are equivalent:
\begin{enumerate}
\item $f\leqW\C_A$,
\item $f$ is non-deterministically computable with advice space $A$.
\end{enumerate}
\end{theorem}
\begin{proof}
We consider the represented spaces $(X,\delta_X)$ and $(Y,\delta_Y)$.
Let $f$ be non-deterministically computable with advice space $A$.
Then there are computable functions $F_1,F_2$ according to Definition~\ref{def:non-deterministic}.
By type conversion and since $\langle\dom(f\delta_X)\times A\rangle\In\dom(F_2)$ we can transfer $F_2$ into a computable function
\[h:\In\Baire\to\CC(A,\IS),p\mapsto(r\mapsto \delta_\IS F_2\langle p,r\rangle).\]
Hence, for each $p\in\dom(f\delta_X)$ the function $h(p)=\chi_{A\setminus A_p}$ is a characteristic function of
the closed set $A_p\in\AA_-(A)$ of successful advices. Here $h$ can also be considered as computable function
of type $h:\In\Baire\to\AA_-(A),p\mapsto A_p$. By condition (1) of Definition~\ref{def:non-deterministic} we
obtain that $A_p\not=\emptyset$ for any $p\in\dom(f\delta_X)$ and by condition (2) we obtain
$\delta_YF_1\langle p,\C_Ah(p)\rangle\In f\delta_X(p)$.
Let $H$ be a computable realizer of $h$.
Then $F_1\langle \id,GH\rangle$ is a realizer of $f$ for any realizer $G$ of $\C_A$
and hence $f\leqW\C_A$.

On the other hand, let $f\leqW\C_A$. Then any realizer $G$ of $\C_A$
computes some realizer $F$ of $f$, i.e.\ there are computable functions
$H,K$ such that for all realizers $G$ of $\C_A$
there is some realizer $F$ of $f$ such that $F(p)=K\langle p,GH(p)\rangle$ for all $p\in\dom(f\delta_X)$.
Now we describe maps $F_1,F_2:\In\Baire\to\Baire$ that satisfy
the conditions of Definition~\ref{def:non-deterministic} for $f$.
For each $p\in\dom(f\delta_X)$ the function $H$ computes a non-empty set $A_p=\psi^A_-H(p)$
and by evaluation there exists a computable function $F_2$ such that
$\delta_\IS F_2\langle p,r\rangle=\chi_{A\setminus A_p}(r)$ for all $p\in\dom(f\delta_X)$ and $r\in A$.
That means to choose $A_p$ as the set of successful advices.
We can also choose $F_1:=K$ and verify the conditions (1) and (2) of Definition~\ref{def:non-deterministic}.
Firstly, it is clear that $A_p\not=\emptyset$ for all $p\in\dom(f\delta_X)$ and hence $F_2$ satisfies condition (1).
Secondly, for each $r\in A_p$ there is a realizer $G$ of $\C_A$ such that $GH(p)=r$ and hence
we obtain $F_1\langle p,r\rangle=K\langle p,GH(p)\rangle=F(p)$ for a realizer $F$ of $f$.
This implies $\delta_YF_1\langle p,r\rangle\in f\delta_X(p)$ and hence condition (2) holds as well.
Altogether $f$ is non-deterministically computable with advice space $A$.
\end{proof}

The main benefit of this characterization of closed choice is that using it
we can easily prove the following theorem that shows that the
advice for compositions can be determined a priori and independently.
We note that due to the fact that Baire space admits a computable and bijective
pairing function, we can always consider $A\times B$ as subspace of Baire space
for any two subspaces $A,B$ of Baire space.

\begin{theorem}[Independent Choice]
\label{thm:independent}
Let $A,B\In\Baire$ and let
$f$ and $g$ be multi-valued functions on represented spaces.
If $f\leqW\C_A$ and $g\leqW\C_B$, then $f\circ g\leqW\C_{A\times B}$.
\end{theorem}
\begin{proof}
We consider represented spaces $(X,\delta_X)$, $(Y,\delta_Y)$ and $(Z,\delta_Z)$.
Let now $f:\In Y\mto Z$ and $g:\In X\mto Y$ be non-deterministically computable with advice spaces
$A$ and $B$, respectively.
Due to Theorem~\ref{thm:non-deterministic} it suffices to show that
$f\circ g$ is non-deterministically computable with advice space $A\times B$.
Intuitively, we can choose an advice $(r,s)\in A\times B$ and use advice $r$
for $f$ and advice $s$ for $g$. More precisely, let $f$ and $g$ be non-deterministically
computable using computable functions $F_1,F_2$ and $G_1,G_2$ according to Definition~\ref{def:non-deterministic},
respectively.
We define $H_1$ and $H_2$ that witness non-deterministic computability of $f\circ g$
with advice space $A\times B$.
We can define a computable $H_1$ by
\[H_1\langle p,\langle r,s\rangle\rangle:=F_1\langle G_1\langle p,s\rangle,r\rangle\]
and there exists a computable $H_2$ such that
\[\delta_\IS H_2\langle p,\langle r,s\rangle\rangle
  =\left\{\begin{array}{ll}
  1 & \mbox{if $\delta_\IS G_2\langle p,s\rangle=1$}\\
  \delta_\IS F_2\langle G_1\langle p,s\rangle,r\rangle & \mbox{otherwise}
\end{array}\right.\]
for all $p\in\dom(fg\delta_X)$ and all $(r,s)\in A\times B$.
Such a computable $H_2$ exists, since $\delta_\IS G_2\langle p,s\rangle=0$ implies
that $\delta_Y G_1\langle p,s\rangle\in g(\delta_X(p))\In\dom(f)$.
Now we verify that $H_1$ and $H_2$ satisfy conditions (1) and (2) of
Definition~\ref{def:non-deterministic} for $f\circ g$.
To this end, let $p\in\dom(fg\delta_X)$.

By condition (1) for $g$ there is an $s\in B$ such that $\delta_\IS G_2\langle p,s\rangle=0$
and hence by condition (2) for $g$ we obtain $\delta_YG_1\langle p,s\rangle\in\dom(f)$.
Hence by condition (1) for $f$ there is an $r\in A$ such that $\delta_\IS F_2\langle G_1\langle p,s\rangle,r\rangle=0$
and thus $\delta_\IS H_2\langle p,\langle r,s\rangle\rangle=0$, which shows that condition (1) also holds for $fg$.

Now let $(r,s)\in A\times B$ be such that $\delta_\IS H_2\langle p,\langle r,s\rangle\rangle=0$.
Then $\delta_\IS G_2\langle p,s\rangle=0$ and $\delta_\IS F_2\langle G_1\langle p,s\rangle,r\rangle=0$.
Hence by conditions (2) for $g$ and $f$ we obtain $\delta_YG_1\langle p,s\rangle\in g\delta_X(p)$
and hence $\delta_ZF_1\langle G_1\langle p,s\rangle,r\rangle\in fg\delta_X(p)$, which proves
condition (2) for $fg$.
\end{proof}

We recall that we call a multi-valued function $h$ on represented spaces
{\em closed under composition} if the principal ideal of $h$ is closed
under composition, i.e.\ if $f\leqW h$ and $g\leqW h$ implies $f\circ g\leqW h$ (for $f$ and $g$
of appropriate type).
It is worth pointing out that closure under composition entails idempotency.

\begin{proposition}
\label{prop:composition-idempotency}
Every multi-valued function $f$ on represented spaces that is closed under composition is also idempotent.
\end{proposition}
\begin{proof}
Let $f:\In X\mto Y$ be a multi-valued function on represented spaces. Then we have
$f\times f=(f\times\id_Y)\circ(\id_X\times f)$ and $f\times\id_Y\leqW f$ and $\id_X\times f\leqW f$.
That is, if $f$ is closed under composition, then $f\times f\leqW f$.
\end{proof}

We get the following consequence of Theorem~\ref{thm:independent}, which is
a strengthening of Corollary~\ref{cor:idempotent} for $A\In\IN^\IN$.

\begin{corollary}[Closure under composition]
\label{cor:composition}
Let $A\In\Baire$ be a subspace of Baire space. If there is a computable surjection
$s:A\to A^2$, then $\C_{A\times A}\leqW\C_A$ and hence $\C_A$
is closed under composition and idempotent.
\end{corollary}

In particular, we can apply this result in the following cases.

\begin{corollary}
The choice functions
$\C_\IN,\C_{\Cantor},\C_{\Baire},\C_{\IN\times\Cantor}$ and hence $\C_\IR$ are closed under composition and idempotent.
\end{corollary}

For most of these functions this was known. However, the proofs in \cite{GM09} and \cite{BG09a}
for the case $\C_{\Cantor}$ are considerably more difficult whereas the Independent Choice Theorem~\ref{thm:independent}
has a simple proof and covers many cases simultaneously.
The results for $\C_{\IN\times\Cantor}$ and $\C_\IR$ seem to be new and are of independent interest.
Closure of non-deterministically computable functions for advice space $\Baire$ was observed
in \cite{Zie07a}.

Now we want to prove that the class of (single-valued)
functions below choice for Baire space $\C_{\Baire}$ is essentially
the class of effectively Borel measurable functions.
It is known that there is no complete Borel measurable function, since
any particular function has to be $\SO{\xi}$--measurable in the Borel
hierarchy for some countable ordinal $\xi$ (see 1G.15 in \cite{Mos80}).
Nevertheless, we will see that choice of Baire space $\C_{\Baire}$ is complete
for Borel measurable functions in a certain sense.
We will say that a function $f:X\to Y$ on computable Polish spaces $X$ and $Y$
is {\em effectively Borel measurable}, if its graph is an effective $\SI{1}$--set
(see Theorem~3E.5 in \cite{Mos80}). Here a subset $A\In X$ of a computable
Polish space $X$ is called {\em effective $\SI{1}$--set}, if there exists
a co-c.e.\ closed set $B\In X\times\Baire$ such that $x\in A\iff(\exists p\in\Baire)(x,p)\in B$.
We will use once again Theorem~\ref{thm:non-deterministic} for the proof.

\begin{theorem}[Choice of Baire space]
Let $X$ and $Y$ be computable Polish spaces and let $f:X\to Y$ be a function.
Then the following are equivalent:
\begin{enumerate}
\item $f\leqW\C_{\Baire}$,
\item $f$ is effectively Borel measurable.
\end{enumerate}
\end{theorem}
\begin{proof}
By Theorem~\ref{thm:non-deterministic} it suffices to show
that $f$ is non-deterministically computable with advice space $\Baire$
if and only if it is effectively Borel measurable.
Since $X$ and $Y$ are Polish, we can assume that we have total computably admissible
representations $\delta_X$ and $\delta_Y$ for $X$ and $Y$, respectively
(see, for instance, Corollary~4.4.12 in \cite{Bra99bx}).

If $f$ is non-deterministically computable with advice space $\Baire$,
then there are computable functions $F_1,F_2$ according to Definition~\ref{def:non-deterministic}.
We obtain for all $(x,y)\in X\times Y$
\begin{eqnarray*}
&& f(x)=y\\
&\iff& (\exists \langle p,r\rangle\in\Baire)(\delta_X(p)=x,\delta_\IS F_2\langle p,r\rangle=0\mbox{ and }\delta_YF_1\langle p,r\rangle=y).
\end{eqnarray*}
Since all involved functions in the matrix of the formula are computable and total, it follows that the
matrix constitutes a co-c.e.\ closed subset of $X\times Y\times\Baire$ in the parameters $(x,y,\langle p,r\rangle)$.
Hence $f$ is effectively Borel measurable.

Let now $f$ be an effectively Borel measurable function.
Then $\graph(f)$ is a $\SI{1}$--set in the effective Borel hierarchy and
there exists a co-c.e.\ closed set $A\In X\times Y\times\Baire$
such that
\[f(x)=y\iff(\exists r\in\Baire)(x,y,r)\in A.\]
We devise a non-deterministic computation for $f$, by defining suitable
computable functions $F_1,F_2$ according to Definition~\ref{def:non-deterministic}.
Firstly, there exists a computable function $F_2$ with
$\delta_\IS F_2\langle p,\langle q,r\rangle\rangle=\chi_{A^{\rm c}}(\delta_X(p),\delta_Y(q),r)$
and we define $F_1\langle p,\langle q,r\rangle\rangle:=q$. Then $F_1$ is computable too
and we obtain
\begin{eqnarray*}
(\exists r\in\Baire)\;\delta_\IS F_2\langle p,\langle q,r\rangle\rangle=0
&\iff& (\exists r\in\Baire)(\delta_X(p),\delta_Y(q),r)\in A\\
&\iff& f\delta_X(p)=\delta_Y(q)
\end{eqnarray*}
and if this condition holds, then we have $\delta_YF_1\langle p,\langle q,r\rangle\rangle=\delta_Y(q)=f\delta_X(p)$.
Altogether, this shows that $F_1,F_2$ satisfy the conditions of Definition~\ref{def:non-deterministic}.
\end{proof}

We note that $\C_{\Baire}$ itself is not Borel measurable, which is not a contradiction, since
it is not a single-valued function defined on a Polish space.
In contrast, the domain of $\C_{\Baire}$ corresponds to the set of ill-founded trees
(i.e.\ trees with at least one infinite branch), which is known to be $\SI{1}$--complete
(see Theorem~27.1 in \cite{Kec95}). We mention that the relativized version of the above proof
leads to the following corollary.

\begin{corollary} Let $X$ and $Y$ be Polish spaces represented by their Cauchy representations
and let $f:X\to Y$ be a function. Then the following are equivalent:
\begin{enumerate}
\item $f\leqW\C_{\Baire}$ with respect to some oracle,
\item $f$ is Borel measurable.
\end{enumerate}
\end{corollary}

Here, reducibility ``with respect to some oracle'' is equivalent to using the continuous version of Weihrauch reducibility.
Now we will consider another model of hypercomputation, namely finitely revising computation
as considered in \cite{Zie07} and as known as computation with finitely many mind changes
in learning theory \cite{BY09}.
A Turing machine that computes with finitely many mind changes or that is finitely revising
can erase its output tape at any stage during its computation and start writing anew, however,
this can be done only finitely often, ensuring that the output is well-defined.
In \cite{Zie07}, the power of finite revising was characterized in terms of an operator
mapping one representation into another. We will define this concept here using the discrete limit
$\lim_\Delta:\In\Baire\to\Baire,\langle p_0,p_1,...\rangle\mapsto\lim_{i\to\infty}p_i$
where the $\Delta$ stands for the discrete topology on $\Baire$ and the limit on the right-hand side
is taken with respect to this topology. That is a sequence $(p_i)_{i\in\IN}$
converges with respect to $\Delta$ if and only if it is eventually constant.
Now we use the discrete limit to define a discrete version of the jump of a representation
(as equivalently considered in \cite{Zie07}).

\begin{definition}[Discrete jump]
Let $(X,\delta)$ be a represented space. Then we define the {\em discrete jump}
of $\delta$ by $\delta^\Delta:=\delta\circ\lim_\Delta$.
\end{definition}

It is easy to see that the following result holds (cf.\ Lemma~3.7 in \cite{Zie07}).

\begin{proposition}[Computability with finitely many mind changes]
\label{prop:mind-changes}
Let $(X,\delta_X)$ and $(Y,\delta_Y)$ be represented spaces and let $f:\In X\mto Y$
be a multi-valued function. Then the following are equivalent:
\begin{enumerate}
\item $f$ is $(\delta_X,\delta_Y)$--computable with finitely many mind changes,
\item $f$ is $(\delta_X,\delta_Y^\Delta)$--computable,
\item $f$ is $(\delta_X^\Delta,\delta_Y^\Delta)$--computable.
\end{enumerate}
\end{proposition}

With this proposition we can produce the following characterization of the
discrete limit and the power of computations with finitely many mind changes in
terms of closed choice, showing that finite revision allows exactly to perform
closed choice in $\mathbb{N}$.

\begin{theorem}[Choice on natural numbers]
\label{thm:choice-natural}
Let $f$ be a multi-valued function on represented spaces.
Then the following are equivalent:
\begin{enumerate}
\item $f \leqW \C_\mathbb{N}$,
\item $f\leqW \lim_\Delta$,
\item $f$ is computable with finitely many mind changes.
\end{enumerate}
\end{theorem}
\begin{proof}
It is easy to see that $\C_\IN$ is computable with finitely many mind changes.
Starting with $n = 0$, the machine outputs a $\delta_\mathbb{N}$ name for $n$ and searches
for $n$ in the input at the same time. If the search is successful, the output is erased, $n$
is increased by $1$, and the machine starts again. A valid input never causes the machine to
erase its output tape infinitely often, and an output can only avoid erasion, if it is a valid result
for $\C_\mathbb{N}$.
Moreover, being computable with finitely many mind changes is preserved downwards
by Weihrauch reducibility (see Lemma~4.4 in \cite{BG09b}) and hence $f\leqW\C_\IN$
implies that $f$ is computable with finitely many mind changes.
Hence (1) implies (3).

Now we assume that $f$ is of type $f:\In X\mto Y$ for represented spaces $(X,\delta_X)$
and $(Y,\delta_Y)$. If $f$ is computable with finitely many mind changes, then
$f$ has a computable $(\delta_X,\delta_Y^\Delta)$--realizer $F$ by
Proposition~\ref{prop:mind-changes},
which means $\delta_Y\circ\lim_\Delta F(p)\in f\delta_X(p)$ for all $p\in\dom(\delta_X)$.
Hence $f\leqW\lim_\Delta$ and (3) implies (2).

In order to prove that (2) implies (1) it suffices to shows $\lim_\Delta\leqW\C_\IN$.
we describe a machine computing a function
$G$ in the following: The input for $G$ is a sequence $(p_n)_{n \in \mathbb{N}}$ with
$p_n \in\Baire$. Now we start to test simultaneously $p_n = p_{n + j}$ for each
$n, j \in \mathbb{N}$. If a contradiction is found, we print $n$ on the output tape.
If $\C_\mathbb{N}$ is applied to the output of $G$, the answer is an index $n_0$,
so that the initial sequence is constant after $n_0$. The remaining task is to output the $n_0$th entry of the sequence.
\end{proof}

We get the following corollary that shows that the discrete limit is equivalent to choice on natural numbers.

\begin{corollary}
\label{cor:discrete-limit}
$\lim_\Delta \equivW\C_\mathbb{N}$.
\end{corollary}

We close by mentioning that the class of operations characterized by choice on
Cantor space is also of independent interest. These functions have been called
{\em weakly computable} in \cite{BG09a} and basically the equivalence of
(1) and (3) below is the definition. With Theorem~\ref{thm:non-deterministic}
we get a characterization of weakly computable functions as non-deterministically computable
ones with advice space $\Cantor$.

\begin{corollary}[Choice on Cantor space]
\label{cor:choice-cantor}
Let $f$ be a multi-valued function on represented spaces.
Then the following are equivalent:
\begin{enumerate}
\item $f \leqW \C_\Cantor$,
\item $f$ is non-deterministically computable with advice space $\Cantor$,
\item $f$ is weakly computable.
\end{enumerate}
\end{corollary}

A surprising omission in our list of classes of computable functions characterized
by closed choice of some space is the class of limit computable functions.
In light of Corollary~\ref{cor:composition} it seems that choice for most
natural spaces will correspond to classes of functions that are closed under
composition, whereas the class of limit computable functions is not closed
under composition (see for instance \cite{Bra05}). Thus, the  following conjecture is plausible.

\begin{conjecture}
There is no represented space $(X,\delta)$ such that $\C_X\equivW\lim$.
\end{conjecture}

At least for Polish spaces $(X,\delta)$ this conjecture follows topologically from Corollary~\ref{cor:dichotomy}.
The closest we can get to a characterization of limit computable functions by
a choice principle of a Polish space is expressed in the following result.

\begin{corollary}[Parallelized choice on natural numbers]
\label{cor:parallel-choice-natural}
Let $f$ be a\linebreak multi-valued function on represented spaces.
Then the following are equivalent:
\begin{enumerate}
\item $f \leqW\widehat{\C_\IN}$,
\item $f$ is limit computable.
\end{enumerate}
\end{corollary}

This follows from $\widehat{\C_\IN}\equiv\lim$ (see Example~\ref{ex:parallelization}) and the
fact that $\lim$ is complete for limit computable functions, see for instance \cite{Bra05}.

\section{A Uniform Low Basis Theorem}

The choice of Cantor space $\C_\Cantor$ is known to be even not non-uniformly computable,
since there is a co-c.e.\ closed set $A\In\Cantor$ that has no computable points
(this can be seen, for instance, using the Kleene tree \cite{Kle52} or Proposition~V.5.25
in \cite{Odi89}).
However, by the Low Basis Theorem of Jockusch and Soare (see Theorem~2.1 in \cite{JS72}
or Proposition~V.5.27 in \cite{Odi89}) any co-c.e.\ closed set $A\In\Cantor$ has a low point, that
is for computable $w$, the set $\psi_-^\Cantor(w)$ always contains a low point.
As shown in \cite{BG09b}, this carries over to all problems below $\C_\mathbb{R}$:
For every computable instance, there is a solution that is low.
We will demonstrate that this result even holds uniformly,
after some necessary definitions have been introduced.

\begin{definition}[Turing jump operator]
Let $(U_n)_{n \in \mathbb{N}}$ be a standard enumeration of the c.e.\ open subsets of
Baire space $\Baire$. Define the \emph{jump operator} $J: \Baire \to \Baire$ by:
$$J(p)(n) = \begin{mycases} 1 & \mbox{if }p \in U_n \\ 0 & \textnormal{otherwise}\end{mycases}$$
\end{definition}

Contrary to its behavior on Turing degrees, as a function on Baire space, the jump is injective.
It even admits a computable inverse $J^{-1}$.
In \cite{Bra07x}, for any representation $\delta$ of some set $X$, a representation $\int \delta$
is defined by $(\int \delta)(p) = \delta(J^{-1}(p))$. Together with the operator $'$ studied in
\cite{Zie07}, where a representation $\delta'$ is defined by $\delta'(p) = \delta(\lim p)$, $\int$
forms a Galois connection, as shown in \cite{Bra07x}.
We define the {\em low representation} $\delta^\vee:=(\int\delta)'$ for any represented space $(X,\delta)$
and if $f:\In X\mto Y$ is a multi-valued map on represented spaces $(X,\delta_X)$ and $(Y,\delta_Y)$,
then $f$ is called {\em low computable}, if $f$ is $(\delta_X,\delta_Y^\vee)$--computable.
In particular,  we will be interested in the low representation $\delta_\Cantor^\vee=\id_\Cantor\circ J^{-1}\circ\lim$ of Cantor space
and the low representation of Baire space $\delta_\Baire^\vee=J^{-1}\circ\lim$, which we also denote by $\Low$.

\begin{lemma}[Low points]
A sequence $p \in \Baire$ is low if and only if it has a computable $\Low$--name.
\end{lemma}
\begin{proof}
By definition, a sequence $p\in\Baire$ is called low, if its Turing jump is Turing reducible
to the halting problem, which is equivalent to $J(p)$ being in the class $\dO{2}$ of the arithmetical hierarchy (see
Proposition~IV.1.16 in \cite{Odi89}).
By Shoenfield's Limit Lemma (see Proposition~IV.1.17 in \cite{Odi89}),
$J(p)\in\dO{2}$ if any only if there exists a computable sequence $q=\langle q_0,q_1,...\rangle\in\Baire$
such that $J(p)=\lim_{i\to\infty}q_i=\lim(q)$, i.e.\ if and only if $p=\Low(q)$.
\end{proof}

Analogously, $p\in\Cantor$ is low if and only if it has a computable $\delta_\Cantor^\vee$--name.
Now we can formulate and prove our uniform low basis theorem, which states that, given an enumeration
of the complement of a non-empty compact subset $A$ of $\Cantor$, we can compute a sequence converging
to the jump of a point $p \in A$.

\begin{theorem}[Uniform Low Basis Theorem]
\label{uniformlowbasistheorem}
$\C_\Cantor$ is low computable.
\end{theorem}
\begin{proof}
We describe a machine $M$ that given a $\psi_-^\Cantor$--name of a compact set $A \subseteq \Cantor$ produces a sequence $(p_m)_{m \in \IN}$ converging to a $\int \id_{\Cantor}$-name of some element of $A$. The input of $M$ is a list enumerating basic open sets exhausting $\Cantor \setminus A$. The complement of the union of the first $m$ of these subsets shall be denoted $A^m$. Likewise, for each $n \in \mathbb{N}$,
we let $U_n^m$ be the union of the first $m$ basic open subsets exhausting $U_n$.
Here, for simplicity, $(U_n)_{n\in\IN}$ is supposed to be a standard enumeration of the c.e.\ open subsets of Cantor space $\{0,1\}^\IN$
and the aforementioned results on the jump and integral are used analogously for Cantor space.

The computation of each $p_m$ can be considered independently, and proceeds as follows.
For each $n \in \IN$, the machine $M$ performs the following tests\footnote{Of course the first test could be subsumed by the second one;
however, since their interpretation is different, we prefer to mention the first test separately.} in the given order:
\begin{enumerate}
\item Does $A^m \subseteq U_n^m$ hold? If the answer is $\textsc{yes}$, the $n$th bit of $p_m$ is $1$.
\item Let $K$ be the set of indexes $i < n$, so that the $i$th bit of $p_m$ is $0$.
        Test $A^m \subseteq U_n^m \cup \bigcup \limits_{i \in K} U_i^m$.
        If the answer is $\textsc{yes}$, the $n$th bit of $p_m$ is $1$.
\item Otherwise, the $n$th bit of $p_m$ is $0$.
\end{enumerate}

All operations are performed on a finite set of basic open sets, either obtained from the input,
or computable by definition. Therefore, each test is decidable. We will first prove that the $p_m$ converge as $m$ goes to infinity.
This is equivalent to showing that each bit of the $p_m$ changes only finitely many times.

The first test is monotone in $m$, as we have $A^{m+1} \subseteq A^m$ and
$U_n^{m} \subseteq U_n^{m+1}$. Thus, if for some $m$ the $n$th bit of $p_m$ was set to
$1$ due to the first test, the $n$th bit of all $p_{m'}$ for $m' > m$ is $1$, too.

Now consider the second test, and assume that all bits $i$ with $i < n$ remain unchanged.
Then, again by the same argument, once the second test yields \textsc{yes} for some $m$,
it will do so for all larger $m'$ as well. The only way for the second test to change the corresponding
bit from $1$ to $0$ is if some smaller bit has been set from $0$ to $1$ previously.

An inductive argument concludes the proof of convergence: The first bit can change at most once,
from $0$ to $1$. All other bits $n$ can change at most once \emph{for each given configuration of the lower bits}.
If only finitely many changes of the bits smaller than $n$ are possible, then there will be only finitely many changes of the $n$th bit.

It remains to show that the $p_m$ actually converge to a correct output $w$. Basically, the first test ensures
that the limit sequence $w$ specifies a point $x \in A$, while the second test ensures that $w$ is a valid $\int \id_\Cantor$-name, i.e. $w \in \dom(J^{-1})$, in the first instance.

To elaborate this, assume $A \subseteq U_n$ for some $n \in \mathbb{N}$.
Then for every $\int \id_\Cantor$-name $w$ with $J^{-1}(w) \in A$ obviously $w(n) = 1$ has to be true.
On the other hand, for $x \notin A$, there is some neighborhood $U$ of $A$ with $x \notin U$.
It is possible to choose $U$ as c.e.\ open (for instance by choosing the complement of some sufficiently small
clopen basic neighborhood of $x$), thus, there is an $n \in \mathbb{N}$ with $A \subseteq U_n$,
but $x \notin U_n$. Thus, having $w(n) = 1$ for each $n \in \mathbb{N}$ with $A \subseteq U_n$
for a $(\int \textnormal{id})$-name $w$ is both necessary and sufficient to ensure $J^{-1}(w) \in A$.

In the next step, we have to show that $A \subseteq U_n$ already guarantees the existence of an
$m \in \mathbb{N}$ with $A^m \subseteq U_n^m$. The other direction is trivial. As $A \subseteq U_n$
is equivalent to $A^{\rm c} \cup U_n = \Cantor$, the basic open sets exhausting $A^{\rm c}$ and $U_n$ are an
open cover of $\Cantor$. Since $\Cantor$ is compact, there has to be a finite subcover. Thus, there is
some $m \in \mathbb{N}$, so that the first $m$ basic open sets in the $\psi_-$-name of $A$ together
with the first $m$ basic open sets listed for $U_n$ already cover $\Cantor$, that is fulfills
$A^m \subseteq U_n^m$, which concludes this part of the proof.

Now we have to show that the second test ensures that the limit sequence $w$ is in the
domain of $\int \id_\Cantor$. This amounts to proving
$$\left ( \bigcap \limits_{i \in \mathbb{N}, w(i) = 1} U_i \right ) \setminus \left ( \bigcup \limits_{j \in \mathbb{N}, w(j) \neq 1} U_j \right ) \neq \emptyset.$$
We note that this difference is automatically a singleton, if non-empty, since any two distinct points
can be separated by two c.e.\ open sets.
We will use the abbreviations
$X := \{i \in \mathbb{N} \mid w(i) = 1 \textnormal{ due to the first test}\}$,
$Y := \{i \in \mathbb{N} \mid w(i) = 1 \textnormal{ due to the second test}\}$ and
$Z := \{i \in \mathbb{N} \mid w(i) = 0\}$, and $U_i^{\rm c} := \{0,1\}^\IN \setminus U_i$.
With this, we have to show:
$$\left ( \bigcap \limits_{i \in X} U_i \right ) \cap \left ( \bigcap \limits_{j \in Y} U_j \right ) \cap \left ( \bigcap \limits_{k \in Z} U_k^{\rm c} \right ) \neq \emptyset.$$
Taking into consideration our results on the first test, this simplifies to:
$$A \cap \left ( \bigcap \limits_{j \in Y} U_j \right ) \cap \left ( \bigcap \limits_{k \in Z} U_k^{\rm c} \right ) \neq \emptyset.$$

Assume that already $A \cap \left ( \bigcap \limits_{k \in Z} U_k^{\rm c} \right ) = \emptyset$
would hold. By the finite intersection property in compact spaces, this implies the existence of a
(smallest) $k_0 \in \mathbb{N}$ with
$A \cap \left ( \bigcap \limits_{k \in Z, k \leq k_0} U_k^{\rm c} \right ) = \emptyset$.
Rearranging the expression yields
$A \subseteq U_{k_0} \cup \bigcup \limits_{k \in Z, k < k_0} U_k$,
so the second test would have been triggered for $k_0$, so $k_0 \notin Z$ follows.
This contradicts the assumption, so we have $A \cap \left ( \bigcap \limits_{k \in Z} U_k^{\rm c} \right ) \neq \emptyset$.

Now we choose some $x \in A \cap \left ( \bigcap \limits_{k \in Z} U_k^{\rm c} \right )$.
Assume $x \notin \bigcap \limits_{j \in Y} U_j$. There has to be some $j_0 \in Y$ with
$x \notin U_{j_0}$. Now $j_0 \in Y$ implies $A \subseteq U_{j_0} \cup \bigcup \limits_{k \in Z, k < j_0} U_k$.
According to the choice of $x$, we have $x \in A$, but $x \notin \bigcup \limits_{k \in Z, k < j_0} U_k$.
This implies $x \in U_{j_0}$, contradicting the assumption. Thus, we have:
$$  A \cap \left ( \bigcap \limits_{k \in Z} U_k^{\rm c} \right )
   = A \cap \left ( \bigcap \limits_{j \in Y} U_j \right ) \cap \left ( \bigcap \limits_{k \in Z} U_k^{\rm c} \right ).$$

As the set on the left is non-empty, so is the set on the right.
With that we know that our Limit-machine always produces a valid output,
that is the jump of some element. We have already established that any
valid output is necessarily correct, and thereby the proof is complete.
\end{proof}

As a corollary of this uniform result we get the known version of the Low Basis Theorem.

\begin{corollary}[Low Basis Theorem of Jockusch and Soare]
Any non\-empty co-c.e.\ closed set $A \subseteq \Cantor$ contains a low point.
\end{corollary}

The property that computable instances always admit low solutions is preserved
under Weihrauch reducibility, as pointed out in \cite{BG09b}.
We will now show that this property also holds uniformly.
This observation invites the question where $\Low=J^{-1} \circ \lim$ is placed in the Weihrauch lattice.
We will start the answer with an obvious corollary to Theorem \ref{uniformlowbasistheorem},
which will then be extended.

\begin{corollary}
\label{corollaryccjlim}
$\C_\Cantor \leqSW \Low$.
\end{corollary}

Now we will lift this observation from compact to locally compact choice. This involves again the same idea as the proof of Proposition \ref{prop:locally-compact-choice}, albeit in a new disguise as the following lemma:

\begin{lemma}
\label{ccbetacncbetahat}
Let $\beta$ be a representation of Cantor space $\Cantor$, and let $\delta_\IN$ be the standard representation of $\mathbb{N}$.
If the multi-valued function $\C_\Cantor:\In\AA_-(\{0,1\}^\IN)\mto\Cantor$ is $(\psi_-^\Cantor, \beta)$--computable,
then the multi-valued function
$\C_{\mathbb{N} \times \Cantor}:\In\AA_-(\IN\times\{0,1\}^\IN)\mto\IN \times \Cantor$ is
$(\psi_-^{\mathbb{N} \times \Cantor}, (\delta_\mathbb{N} \times \beta)^\Delta)$--computable.
\end{lemma}
\begin{proof}
We describe a machine solving the latter task. Given a $\psi_-^{\mathbb{N} \times \Cantor}$-name of a closed set $A\In\IN\times\Cantor$, it produces a sequence $(p_m)_{m \in \IN}$. Again we use $A^m$ to denote the complement of the union of the first $m$ basic open sets listed in the input.

As $(\{n\} \times \Cantor) \cap A^m = \emptyset?$ is decidable and we have $A^m \neq \emptyset$, we can compute $n_m = \min \{n \in \IN \mid (\{n\} \times \Cantor) \cap A^m \neq \emptyset\}$. Using these values, the output sequence shall be of the form $p_m = \langle\delta_\IN^{-1}(n_m), q_{n_m}\rangle$. Note that $n_m$ will be eventually constant as $m$ goes to infinity, hence the same is true for the $p_m$.

The values $q_{n_m}$ are computed as follows. A machine computing $\C_\Cantor:\In (\AA_-(\{0,1\}^\IN), \psi_-^\Cantor) \mto (\Cantor, \beta)$ is simulated on input denoting \newline $\pr_2((\{n_m\} \times \Cantor) \cap A)$ for $k$ steps, as long as $(\{n_m\} \times \Cantor) \cap A^k \neq \emptyset$ for $k \in \IN$. If a $k$ is reached with $(\{n_m\} \times \Cantor) \cap A^k = \emptyset$, the sequence $q_{n_m}$ will be continued by $0$s.

If $n_m$ has reached its final value for $m_0$, then $q_{n_{m_0}}$ will be a $\beta$-name for some $w$ with $n_{m_0} \times w \in A$; this is sufficient to ensure that the overall output of the described computation is a $(\delta_\mathbb{N} \times \beta)^\Delta$-name of $n_{m_0} \times w \in A$.
\end{proof}

As a consequence we obtain that $\Low$ is strictly above locally compact choice.

\begin{theorem}
\label{thm:low-reals}
$\C_{\mathbb{N} \times \Cantor} \leqSW \Low$ and $\Low \nleqW \C_{\mathbb{N} \times \Cantor}$.
\end{theorem}
\begin{proof}
To show the reduction, we make use of Theorem \ref{uniformlowbasistheorem} together
with Lemma \ref{ccbetacncbetahat} and the observation that $\beta^{\vee\Delta}\equiv\beta^{\vee}$ for any representation $\beta$.
To see $\Low \nleqW \C_{\mathbb{N} \times \Cantor}$, observe that $\Low$
is single-valued. Therefore, the assumption of the contrary together with Corollary \ref{cor:quotient-reals}
would imply $\Low \leqW \C_\mathbb{N}$. By transitivity and Corollary \ref{corollaryccjlim}
we get $\C_\Cantor \leqW \C_\mathbb{N}$. As shown in \cite{BG09b}, the latter is wrong, providing
the sought contradiction.
\end{proof}

As $J^{-1}$ is computable, the upper bound $\Low=J^{-1} \circ \lim \leqW \lim$ is obtained directly.
As $\lim$ maps some computable inputs to non-low outputs, we even have $\Low <_W \lim$.
With this, we have determined precisely the place of $J^{-1} \circ \lim$ in the diagram provided
in Figure~\ref{fig:choice}.

A question regarding the Weihrauch degree of $\Low$ that is left open by the results presented
so far is its behavior under products. Remarkable consequences of the following answers are that the low
real numbers do not form a field, and that the integral does not commute with products.

\begin{theorem}
\label{thm:L-not-idempotent}
$\Low \lW \Low \times \Low$.
\end{theorem}
\begin{proof}
By a result of Spector (see \cite{Spe56} or Proposition V.2.26 in \cite{Odi89})
there are sequences $a, b \in \{0, 1\}^\mathbb{N}$,
so that both $a$ and $b$ are low, but $\langle a, b \rangle$ is not low.
Since $a$ and $b$ are low, $J(a)$ and $J(b)$ are Turing reducible to the halting problem,
there are computable sequences $\langle a_0,a_1,...\rangle$ and $\langle b_0,b_1,...\rangle$
with $\lim_{i\to\infty} a_i = J(a)$ and $\lim_{i\to\infty} b_i = J(b)$.
Then $(\langle a_0,a_1,...\rangle,\langle b_0,b_1,...\rangle)$ is computable, and we have
$(J^{-1} \circ \lim\times J^{-1} \circ \lim)(\langle a_0,a_1,...\rangle,\langle b_0,b_1,...\rangle) = (a, b)$.
Thus, $(J^{-1} \circ \lim) \times (J^{-1} \circ \lim)$ can map a computable input to an output that is not low.
\end{proof}

In other words, this means that $\Low$ is not idempotent. However, it has a different property.
We call a function $T:\In\Baire\mto\Baire$ a {\em jump operator}, if for all computable
functions $F:\In\Baire\to\Baire$ there exists a computable function $G:\In\Baire\to\Baire$
such that $F\circ T=T\circ G$.  This notion has been introduced in \cite{Bre09x} (for single-valued functions)
and using this terminology the following has been proved in \cite{Bra07x}.

\begin{lemma}
\label{lem:lim-jump}
The limit $\lim$ and the inverse of the Turing jump $J^{-1}$ are jump operators
and hence $\Low$ is also a jump operator.
\end{lemma}

Now we can formulate our main characterization of low computability.

\begin{theorem}[Low computability]
\label{thm:low-computability}
Let $f$ be a multi-valued function on represented spaces.
Then the following are equivalent:
\begin{enumerate}
\item $f\leqSW\Low$,
\item $f$ is low computable.
\end{enumerate}
\end{theorem}
\begin{proof}
We consider the represented spaces $(X,\delta_X)$ and $(Y,\delta_Y)$.
If $f:\In X\mto Y$ is low computable, then there is a computable
realizer $F$ such that $\delta_Y^\vee\circ F(p)\in f\delta_X(p)$ for all
$p\in\dom(f\delta_X)$. Since $\delta_Y^\vee\circ F=\delta_Y\circ\Low\circ F$,
this means that $\Low\circ F$ is a $(\delta_X,\delta_Y)$--realizer
of $f$ and hence $f\leqSW\Low$.
If, on the other hand, $f\leqSW\Low$, then there are computable functions
$H,K$ such that $F=H\Low K$ is a $(\delta_X,\delta_Y)$--realizer of $f$.
By Lemma~\ref{lem:lim-jump} there is a computable function $L$ such that $H\Low=\Low L$
and hence $F=\Low LK$ and $LK$ is a $(\delta_X,\delta_Y^\vee)$--realizer of $f$.
\end{proof}

Next we want to show that certain choice principles are cylinders.
We recall that a multi-valued map $f$ on represented spaces is called a {\em cylinder},
if $\id\times f\leqSW f$. For cylinders $f$ we have $g\leqSW f\iff g\leqW f$ (see \cite{BG09a}).
It has already been proved in \cite{BG09a} that $\C_\Cantor$ is a cylinder, here we present
another proof that can be directly transferred to $\C_{\IN\times\Cantor}$.

\begin{proposition}
\label{prop:cylinder}
$\C_{\Cantor}$ and $\C_{\IN\times\Cantor}$ are cylinders.
\end{proposition}
\begin{proof}
There is a computable embedding
\[\iota:\Baire\to\Cantor,p\mapsto 01^{p(0)+1}01^{p(1)+1}...\]
and using this embedding we get $\id_\Baire(p)=\iota^{-1}\circ\C_\Cantor\circ\inj_\Cantor\circ\iota(p)$
and hence $\id_\Baire\leqSW\C_\Cantor$.
The proofs of Propositions~\ref{prop:product} and \ref{prop:surjection} even
show strong Weihrauch reducibility. Hence, using a computable surjective pairing function
$\pi:\Cantor\to\Cantor\times\Cantor$ one obtains
\[\id_\Baire\times\C_\Cantor\leqSW\C_\Cantor\times\C_\Cantor\leqSW\C_{\Cantor\times\Cantor}\leqSW\C_\Cantor.\]
Hence $\C_{\Cantor}$ is a cylinder. The fact that $\C_{\IN\times\Cantor}$ is a cylinder can be proved analogously.
\end{proof}

Together with Propositions~\ref{prop:locally-compact-choice} and \ref{prop:cylinder}, Corollary~\ref{cor:real-choice}
and Theorems~\ref{thm:low-reals} and \ref{thm:low-computability} we obtain the following corollary.

\begin{corollary}
If $X$ is a computable $K_\sigma$--space, then $\C_X$ is low computable.
\end{corollary}

This applies, in particular, to $\C_\IN$, $\C_\Cantor$ and $\C_\IR$.
We also obtain the following generalization of the non-uniform Low Basis Theorem of Jockusch
and Soare. The case $\C_\IR$ was already treated as Theorem~4.7 in \cite{BG09b}.

\begin{corollary}[Low Basis Theorem]
If $X$ is a computable $K_\sigma$--space, then any non-empty co-c.e.\ closed set $A\In X$
contains a low point.
\end{corollary}

Together with Corollary~\ref{cor:choice-cantor} and Theorem~\ref{thm:choice-natural}
we obtain that the class of low computable functions contains several others.

\begin{corollary}
Any multi-valued function $f$ on represented spaces that is computable with finitely many
mind changes or weakly computable is also low computable.
\end{corollary}

We mention that one gets consequences as the following.

\begin{corollary}
The Brouwer Fixed Point Theorem $\BFT$ is low computable.
\end{corollary}

Here, $\BFT:\CC([0,1]^n,[0,1]^n)\mto[0,1]^n$ is the multi-valued map with
$\BFT(f):=\{x\in[0,1]^n:f(x)=x\}$. In \cite{BG09b} it was already proved that any
computable function $f:[0,1]^n\to[0,1]^n$ has a low fixed point and that the
Brouwer Fixed Point Theorem is weakly computable. The above property
is a uniform version of the former fact.
The benefit of having uniform results is highlighted by the following result.

An interesting property of the class of low computable functions is that
if they are composed with limit computable functions from the left, then
one obtains a limit computable function again. This is in contrast to the
fact that the limit computable functions themselves are not closed under composition.

\begin{proposition}[Composition]
\label{prop:low-composition}
Let $f:\In X\mto Y$ and $g:\In Y\mto Z$ be multi-valued functions on represented spaces.
If $f$ is low computable and $g$ is limit computable, then $g\circ f$ is limit computable.
If $f$ and $g$ are both low computable, then $g\circ f$ is low computable.
\end{proposition}
\begin{proof}
We use the represented spaces $(X,\delta_X)$, $(Y,\delta_Y)$ and $(Z,\delta_Z)$.
We exploit the fact that integral and derivative of representations form a Galois connection (see \cite{Bra07x}).
That $g$ is limit computable means that it is $(\delta_Y,\delta_Z')$--computable,
which is equivalent to $g$ being $(\int\delta_Y,\delta_Z)$--computable
and that $f$ is low computable means that it is $(\delta_X,\delta_Y^\vee)$--computable,
which is equivalent to $f$ being $(\int\delta_X,\int\delta_Y)$--computable.
It follows that $g\circ f$ is $(\int\delta_X,\delta_Z)$--computable, which
is equivalent to $g\circ f$ being limit computable.
Analogously, if $f$ and $g$ are both low computable, then it follows that
$g\circ f$ is $(\int\delta_X,\int\delta_Z)$--computable, which is equivalent
to $g\circ f$ being low computable.
\end{proof}

It can easily be seen that the composition $g\circ f$ of a limit computable $f$ even with a
$g$ that is computable with finitely many mind changes is not necessarily limit computable.
The class of low computable functions is the largest known class with the stability property expressed
in Proposition~\ref{prop:low-composition}.

We note that $\LL$ cannot be closed under composition by Theorem~\ref{thm:L-not-idempotent}
and Proposition~\ref{prop:composition-idempotency}. Hence strict Weihrauch reducibility cannot be replaced by ordinary
Weihrauch reducibility in Theorem~\ref{thm:low-computability}.

\section{The Jump Topology}

Connecting to the results of Section \ref{sec:non-determinism}, it seems reasonable to inquire whether other
interesting Weihrauch degrees can be characterized by restrictions of the limit operation $\lim$ of Baire space $\IN^\IN$.
Since all such restrictions are single-valued, neither $\C_\Cantor$ nor $\C_\mathbb{R}$ can be equivalent to such an
operation, as their Weihrauch degrees do not contain any single-valued functions by Corollaries~\ref{cor:quotient-cantor}
and \ref{cor:quotient-reals}.
In the remainder of this section, we will study the limit operator $\lim_J$ with respect to the
initial topology of the jump $J$. Like $\lim_\Delta$ we will consider this operation as
an operation with respect to Baire space (with the identity as standard representation).

Initially, we suspected that $\lim_J$ might be equivalent to $\Low=J^{-1} \circ \lim$,
but this is only true topologically, as we will show in Theorem~\ref{thm:jump-low-top}.
Computationally, the contrary result is given below (see Theorem~\ref{thm:cantor-jump}).
It turned out that the initial topology of the jump is identical to the $\Pi$--topology studied by Miller \cite[Chapter IV]{Mil02a}.

\begin{theorem}
\label{thm:pi-topology}
The initial topology of $J$ is generated by the co-c.e.\ closed sets (that is identical to the $\Pi$--topology).
\end{theorem}
\begin{proof}
As every basic set of the form $w\mathbb{N}^\mathbb{N}$ for some finite $w$ is co-c.e.\ closed,
every set that is open in the ordinary Baire topology is also open in the $\Pi$--topology.
Now consider the preimage:
$$J^{-1}(w\mathbb{N}^\mathbb{N}) = \left ( \bigcap \limits_{i < |w|, w(i) = 1} U_i \right ) \cap \left ( \bigcap \limits_{j < |w|, w(j) = 0} U_j^{\rm c} \right ).$$
In the $\Pi$--topology, this is an intersection of finitely many open sets, and therefore open.
As the Baire topology is generated by sets of the form $w\mathbb{N}^\mathbb{N}$,
this shows that the jump $J$ is continuous with the $\Pi$--topology on its domain and the Baire topology on its codomain.
This is equivalent to the inclusion of the initial topology of $J$ in the $\Pi$--topology.

For the other inclusion, fix some co-c.e.\ closed set $U_n^{\rm c}$.
We have
$$U_n^{\rm c} = \bigcup \limits_{w \in \mathbb{N}^{n}} J^{-1}(w0\mathbb{N}^\mathbb{N}),$$
so $U_n^{\rm c}$ is open in the initial topology of the jump. This concludes the proof.
\end{proof}

A sequence $(p_n)_{n \in \mathbb{N}}$ in $\IN^\IN$ converges to $p\in\IN^\IN$ regarding the
$\Pi$--topology, if $(J(p_n))_{n \in \mathbb{N}}$ converges to $J(p)$ in Baire space.
The limit value $p$ cannot be left out here:
There is a sequence $(p_n)_{n \in \mathbb{N}}$, so that $(J(p_n))_{n \in \mathbb{N}}$ converges
in Baire space, but not to some element of the range of $J$, as the range of $J$ is not closed in Baire space.
The above description of the convergence relation of the $\Pi$--topology implies
\[\lim\nolimits_J=J^{-1}\circ\lim\circ J^\IN=\Low\circ J^\IN,\]
with $J^\IN\langle p_0,p_1,p_2,...\rangle:=\langle J(p_0),J(p_1),J(p_2),...\rangle$.

In order to understand the computability aspects of $\lim_J$, we would like to know
which points are limits of computable sequences with respect to the $\Pi$--topology.
We introduce a name for these points.

\begin{definition}
A point $p\in\IN^\IN$ is called {\em limit computable in the jump}, if there is a computable
sequence $(p_n)_{n\in\IN}$ in $\IN^\IN$ such that $\lim_{n\to\infty}J(p_n)=J(p)$.
\end{definition}

Here the limit is understood with respect to the ordinary Baire topology and by continuity of $J^{-1}$
we automatically obtain $\lim_{n\to\infty} p_n=p$.
Some necessary properties of points $p$ that are limit computable in the jump are clear.
For one, they are limit computable and secondly they are in the closure of the set
of computable points with respect to the $\Pi$--topology. These points are called
{\em unavoidable} following Kalantari and Welch (see \cite{KW03} and \cite{Mil02a}).

Another observation is that all limit computable $1$-generics are limit computable in the jump.
We recall that a point $p\in\IN^\IN$ is called {\em $1$--generic}, if for all $n\in\IN$ there
exists a finite word $w\prefix p$ such that either $w\IN^\IN\In U_n$ or $w\IN^\IN\cap U_n=\emptyset$
(see \cite{Nie09}).
Here $(U_n)_{n\in\IN}$ denotes the computable standard enumeration of all c.e.\ open subsets of Baire space
that was used to define the Turing jump $J$.
The definition directly implies the following observation.

\begin{lemma}
\label{lem:1-generic}
The Turing jump operator $J:\IN^\IN\to\IN^\IN$ is continuous in $p\in\IN^\IN$
if and only if $p$ is $1$--generic.
\end{lemma}

Using this lemma, we obtain the following sufficient condition for limit computability in the jump.

\begin{proposition}
\label{prop:jump-generic}
If $p$ is $1$--generic and limit computable, then $p$ is limit computable in the jump.
\end{proposition}
\begin{proof}
If $p$ is limit computable, then there is a computable sequence $(p_n)_{n\in\IN}$ that converges to $p$.
If $p$ is $1$--generic, then $(J(p_n))_{n\in\IN}$ also converges to $J(p)$ according to Lemma~\ref{lem:1-generic}.
This means that $p$ is limit computable in the jump.
\end{proof}

It is known that there is a $1$--generic and limit computable $p\in\IN^\IN$ (see Theorem~1.8.52 in \cite{Nie09}).
Moreover, a $1$--generic cannot be computable (see for instance Proposition~XI.2.3 in \cite{Odi99}).
Hence, it follows that $\lim_J$ maps some computable input to a non-computable output
and hence it is not non-uniformly computable.

It will follow from Proposition~\ref{prop:jump-cone} below that points which are non-computable and limit computable in the jump
are not necessarily $1$--generic. However, they seem to share a lot of properties with the class of
limit computable $1$--generics. As one such property we prove that points which are limit computable in the jump
do not bound diagonally non-computable functions.
A total function $f:\IN\to\IN$ is called {\em diagonally non-computable}
if $f(i)\not=\varphi_i(i)$ for all $i\in\IN$ (that means either $\varphi_i(i)$ does not exist or otherwise the two values are not equal).
Here $\varphi$ denotes some standard G\"odel numbering of the partial computable functions $g:\In\IN\to\IN$.
Diagonally non-computable functions are, in particular, not computable.
As we will show below, our following proposition is related to the known result that
$1$--generics do not bound diagonally non-computable functions (due to Demuth and Ku\v{c}era, see Corollary~9 in \cite{DK87}).
The proof is inspired by Nies (see Exercise~4.1.6 in \cite{Nie09}).

\begin{proposition}
\label{prop:dnr-jump}
Let $f$ be diagonally non-computable and let $p$ be limit computable in the jump.
Then $f\nleqT p$.
\end{proposition}
\begin{proof}
Let $f$ be diagonally non-computable and let $p$ be limit computable in the jump.
Let us assume that $f\leqT p$.
Then there is a computable function $F:\In\IN^\IN\to\IN^\IN$ such that $F(p)=f$ and there
is a computable sequence $(p_n)_{n\in\IN}$ which converges to $p$ in the $\Pi$--topology.
Since $F$ is computable, there is a Turing machine $M$ that computes $F$.
Let us denote by $F_M(r)(n)$ the $n$--th symbol written by this machine $M$ upon input $r$,
irrespectively of whether $r\in\dom(F)$. Then the set
\[U:=\{r\in\IN^\IN:(\exists i\in\IN)\;(F_M(r)(i)=\varphi_i(i)\mbox{ and }i\in\dom(\varphi_i))\}\]
is c.e.\ open and since $f=F(p)$ is diagonally non-computable, it follows that $p\not\in U$.
Since $(p_n)_{n\in\IN}$ converges to $p$ in the $\Pi$--topology and the complement of $U$ is
open in the $\Pi$--topology by Theorem~\ref{thm:pi-topology},
it follows that $p_n\not\in U$ for all $n\geq m$ with some fixed $m\in\IN$.
Since $f=F(p)$ is total and $(p_n)_{n\in\IN}$ converges to $p$, there must be an $n\geq m$
for each $i\in\IN$ such that $F_M(p_n)(i)$ exists. Since $(p_n)_{n\in\IN}$ is computable,
we can even find such an $n$ effectively, i.e.\ there is a computable function $s:\IN\to\IN$ such that $F_M(p_{s(i)})(i)$
exists and $s(i)\geq m$ for all $i\in\IN$. Since $p_{s(i)}\not\in U$, we obtain
$F_M(p_{s(i)})(i)\not=\varphi_i(i)$. But that means that $g(i):=F_M(p_{s(i)})(i)$ defines a total computable
function $g:\IN\to\IN$ that is diagonally non-computable, which is a contradiction!
\end{proof}

From this result we can directly conclude that choice on Cantor space $\C_\Cantor$ is not reducible to $\lim_J$.
A function $f$ is called {\em two-valued diagonally non-computable} if it is diagonally non-computable
and $\range(f)\In\{0,1\}$. It is known that the set of all such functions is co-c.e.\ closed in Cantor space
$\{0,1\}^\IN$ (see Fact~1.8.31 in \cite{Nie09}).

\begin{theorem}
\label{thm:cantor-jump}
We obtain $\C_\Cantor\nleqW\lim_J$.
\end{theorem}
\begin{proof}
Let us assume to the contrary that $\C_\Cantor\leqW\lim_J$.
Then there are computable functions $H,K$ such that $H\langle p,\lim_JK(p)\rangle\in\C_\Cantor\psi_-(p)$
for all $p$ in the domain of the right-hand side.
It is known and easy to see that the set
\[A:=\{f\in\{0,1\}^\IN:\mbox{$f$ is two-valued diagonally non-computable}\}\]
is a co-c.e.\ closed set. Hence, there is a computable $p$ such that $A=\psi_-(p)$
and we obtain that $f:=H\langle p,\lim_JK(p)\rangle$ is diagonally non-computable.
Hence $K(p)$ is computable and $q:=\lim_JK(p)$ is limit computable in the jump.
Moreover, $f\leqT q$, which contradicts Proposition~\ref{prop:dnr-jump}.
\end{proof}

Next we prove that $\lim_J$ is low computable.

\begin{theorem}
\label{thm:jump-low}
We obtain $\lim_J\lSW\Low$.
\end{theorem}
\begin{proof}
We use the computable standard enumeration $(U_n)_{n\in\IN}$ of c.e.\ open
subsets $U_n\In\IN^\IN$ that was used to define the Turing jump operator $J$.
By $U_n^m$ we denote the union of the first $m$ basic clopen balls in the union
that constitutes $U_n$.
We define a function $F:\IN^\IN\to\IN^\IN$ by
$F\langle p_0,p_1,p_2,...\rangle:=\langle q_0,q_1,q_2,...\rangle$
with
\[q_{\langle k,m\rangle}(n):=\left\{\begin{array}{ll}
  1 & \mbox{if $p_k\in U_n^m$}\\
  0 & \mbox{otherwise}
\end{array}\right.\]
Since the property $p_k\in U_n^m$ is decidable in the input sequence and the parameters $k,n,m$,
it follows that $F$ is computable. We claim that $\lim_J=\Low\circ F$.
Let $(p_k)_{k\in\IN}$ and $p$ be such that $\lim_{k\to\infty}J(p_k)=J(p)$.
Then also $\lim_{k\to\infty}p_k=p$. Let $(q_i)_{i\in\IN}$ be the corresponding
output of $F$.
Let us assume that $J(p)(n)=1$ for some $n\in\IN$, i.e.\ $p\in U_n$.
Then $p\in U_n^m$ for all sufficiently large $m$ and hence $p_k\in U_n^m$
for all sufficiently large $k,m$. This implies that $q_{\langle m,k\rangle}(n)=1$
for sufficiently large $\langle m,k\rangle$. Let us now assume that $J(p)(n)=0$,
i.e.\ $p\not\in U_n$. Since the complement of $U_n$ is co-c.e. closed and hence
open in the $\Pi$--topology, this implies that $p_k\not\in U_n$ for all
sufficiently large $k$. In particular, $p_k\not\in U_n^m$ for all $m$ and all
sufficiently large $k$. This implies that $q_{\langle m,k\rangle}(n)=0$
for all sufficiently large $\langle m,k\rangle$.
Altogether, this means $\Low\circ F\langle p_0,p_1,p_2,...\rangle=J^{-1}\circ\lim\langle q_0,q_1,q_2,...\rangle=p$,
as desired. By Theorem~\ref{thm:cantor-jump} the reduction is strict.
\end{proof}

As a corollary we obtain the following.

\begin{corollary}
All $p\in\IN^\IN$ which are limit computable in the jump are also low.
\end{corollary}

This is another property that points which are limit computable in the jump share
with limit computable $1$--generics (see Proposition~XI.2.3.2 in \cite{Odi99}).
Another straightforward observation is the following.

\begin{corollary}
We obtain $\lim\nolimits_\Delta\lW\lim\nolimits_J\lW\lim$.
\end{corollary}

Since the corresponding topologies are included in each other in the converse order,
each limit operation in this sequence is just a restriction of the next one.
This implies the positive part of the reduction chain. The first reduction is strict, since $\lim_\Delta$ is non-uniformly
computable and $\lim_J$ is not (as observed after Proposition~\ref{prop:jump-generic}).
The second reduction is strict since $\C_\Cantor$ is reducible
to $\lim$, but not to $\lim_J$ (by Theorem~\ref{thm:cantor-jump}).

In light of Theorem~\ref{thm:jump-low} it might be surprising that topologically $\lim_J$ turns
out to be equivalent to $\Low$.

\begin{theorem}
\label{thm:jump-low-top}
We obtain $\lim_J\equivSW\Low$ with respect to some oracle.
\end{theorem}
\begin{proof}
By Theorem~\ref{thm:jump-low} it is clear that $\lim_J\leqSW\Low$. We need to show the reverse reduction
with respect to some oracle.

Let $(U_i)_{i\in\IN}$ be the standard enumeration of c.e.\ open sets used to define
the jump operator $J$. Given a finite word $w=w_0...w_n\in\IN^*$ we use the sets
\[A_{w,i}:=\left\{\begin{array}{ll}
U_i & \mbox{if $w_i\not=0$}\\
\IN^\IN\setminus U_i & \mbox{otherwise}
\end{array}\right.\]
for all $i=0,...,n$. Moreover, we set $A_w:=\bigcap_{i=0}^nA_{w,i}$. Now we define inductively
a function $f:\IN^*\to\IN^\IN$ by $f(\varepsilon):=\widehat{0}$ for the empty word $\varepsilon$
and for $w:=w_0...w_{n+1}$ we select $f(w)\in A_w$ if $A_w\not=\emptyset$ and $f(w):=f(w_0...w_n)$
otherwise. By the Axiom of Choice such a function $f$ exists and we use it as an oracle in the following.
Given a sequence $p=\langle p_0,p_1,p_2,...\rangle\in\dom(\Low)$ we let
\[F(p):=\langle q_0,q_1,q_2,...\rangle\mbox{ with }q_i:=f(p_i[i]),\]
where $p_i[j]=p_i(0)...p_i(j)$ denotes the prefix of $p_i$ of length $j+1$.
It is clear that $F$ is computable in the oracle $f$.

We claim that $\lim_JF(p)=\Low(p)$ for all $p\in\dom(\Low)$.
Given a sequence $p=\langle p_0,p_1,p_2,...\rangle\in\dom(\Low)$ it follows
that the sequence $(p_i)_{i\in\IN}$ converges to $J\Low(p)$ in the usual Baire topology.
We consider $(q_i)_{i\in\IN}$ with $q_i=f(p_i[i])$ as above. Let $n\in\IN$. Then there is an $i\geq n$
such that
\[p_j(m)=1\iff\Low(p)\in U_m\]
for all $j\geq i$ and $m\leq n$.
In this situation $A_{p_i[n]}\not=\emptyset$ since $\Low(p)\in A_{p_i[n]}$ and hence $q_j\in A_{p_i[n]}$ for all $j\geq i$
by definition of $f$. In particular,
\[q_j\in U_n\iff\Low(p)\in U_n\]
for all $j\geq i$.
This means that $(q_j)_{j\in\IN}$ converges to $\Low(p)$ in the $\Pi$--topology and hence
$\lim_JF(p)=\Low(p)$.
\end{proof}

As a last result on limit computability in the limit we prove that
this class of points is closed under total computable functions.

\begin{proposition}
\label{prop:jump-cone}
Let $p,q\in\IN^\IN$ be such that $F(p)=q$ for some total computable function $F:\IN^\IN\to\IN^\IN$.
If $p$ is limit computable in the jump, then $q$ is limit computable
in the jump too.
\end{proposition}
\begin{proof}
Let $F:\IN^\IN\to\IN^\IN$ be some total computable function such that
$F(p)=q$. Hence $JF$ is limit computable and hence there exists a computable $G:\In\IN^\IN\to\IN^\IN$
such that $JF=GJ$ by Lemma~\ref{lem:lim-jump}.
If $p$ is limit computable in the jump, then there is a computable sequence
$(p_n)_{n\in\IN}$ such that $(J(p_n))_{n\in\IN}$ converges to $J(p)$. Since $G$ is continuous,
we obtain that $(GJ(p_n))_{n\in\IN}$ converges to $GJ(p)$, which implies that $(JF(p_n))_{n\in\IN}$ converges
to $JF(p)$.
Since $F$ is computable, it follows that $(F(p_n))_{n\in\IN}$ is computable and
this means that $q=F(p)$ is limit computable in the jump.
\end{proof}

From this result it follows that $p$ which are non-computable and limit computable in the jump are not
necessarily $1$--generic. For instance, for each limit computable $1$--generic $p$ we have that
$\langle\widehat{0},p\rangle$ is limit computable in the jump and
non-computable,  but it is not $1$--generic, since no finite prefix proves that it does belong to the
co-c.e.\ closed set $\{\langle\widehat{0},q\rangle:q\in\IN^\IN\}$.
So far, we have no example of a point that is limit-computable in the jump and not below
a $1$--generic with respect to truth-table reducibility. It would be useful to clarify the
relation between $1$--generics and points that are limit computable in the jump somewhat further.

The $\Pi$--topology shows further interesting behavior. If $p$ is computable in Baire space,
then it is isolated regarding the $\Pi$--topology, that is the singletons $\{p\}$ with computable $p$ are clopen.
We characterize the singletons $\{p\}$ that are clopen in the $\Pi$--topology.

\begin{lemma}
\label{lem:clopen}
Let $p\in\Baire$. Then $\{p\}$ is clopen in the $\Pi$--topology if and only if $\{p\}$ is co-c.e.\ closed in Baire space.
\end{lemma}
\begin{proof}
Since the $\Pi$--topology includes the ordinary Baire topology, it is clear that all singletons
$\{p\}$ are closed in the $\Pi$--topology.
If $\{p\}$ is co-c.e.\ closed in Baire space, then $\{p\}$ is also open in the $\Pi$--topology
(since this topology is generated by the co-c.e.\ closed sets).
Let now $\{p\}$ be open in the $\Pi$--topology.
Then there is a finite prefix $w\prefix J(p)$ such that $\{p\}=J^{-1}(w\IN^\IN)$.
Similarly to the proof of Theorem~\ref{thm:pi-topology} we obtain
$$\{p\}=J^{-1}(w\mathbb{N}^\mathbb{N}) = \left ( \bigcap \limits_{i < |w|, w(i) = 1} U_i \right ) \cap \left ( \bigcap \limits_{j < |w|, w(j) = 0} U_j^{\rm c} \right ).$$
However, in this case the open sets $U_i$ with $w(i)=1$ can be replaced by clopen balls, since
$\{p\}$ is a singleton. Altogether, this implies that $\{p\}$ can be written as a finite intersection
of co-c.e.\ closed sets and hence it is co-c.e.\ closed.
\end{proof}

It is easy to see that there are co-c.e.\ closed singletons $\{p\}$ with non-computable $p\in\IN^\IN$.
Co-c.e.\ closed singletons $\{p\}$ can even be such that $p$ is not arithmetical
(see Propositions~1.8.62 and 1.8.70 in \cite{Nie09}).
Lemma~\ref{lem:clopen} implies that the set of computable points is open in the $\Pi$--topology, although not effectively so. In general,
a set $O$ is c.e.\ open in the $\Pi$--topology, if and only if it is effectively $F_\sigma$ in
Baire space, in turn, a set $A$ is co-c.e.\ closed in the $\Pi$--topology, if and only if it is
effectively $G_\delta$ in Baire space. In particular, the Martin-L\"{o}f random points form a c.e.\ open
set (and a proper subset of the open set of avoidable points).
While the $\Pi$--topology \emph{makes everything easier} when considering sets, it
\emph{makes everything more complicated} when considering points: a point is Turing reducible
to the $n$--th jump of the empty set in Baire space if and only if it has a name in the $\Pi$--space
that is Turing reducible to the $n+1$--st jump of the empty set.
This contrary behavior of points and sets is based on the fact that points are mapped forwards
and sets are mapped backwards.

We close by mentioning another property of the Turing jump $J$. The Galois connection between
the Turing jump $J$ and the limit $\lim$ cannot be extended to the continuous category.
That is, we get the following counterexample, which shows that the inverse $J^{-1}$ is
not a ``topological jump operator'' (see Lemma~\ref{lem:lim-jump}).

\begin{proposition}
There exists a total continuous function $F:\Baire\to\Baire$ such that
there is no continuous function $G:\In\Baire\to\Baire$ with $FJ^{-1}=J^{-1}G$.
\end{proposition}
\begin{proof}
Let $c,d\in\Baire$ be such that $c$ is computable and $\{d\}$ is not co-c.e.\ closed.
Then according to Lemma~\ref{lem:clopen} $\{c\}$ is clopen and $\{d\}$ is not clopen in the $\Pi$--topology.
Since by Theorem~\ref{thm:pi-topology} the $\Pi$--topology is just the initial topology of the jump $J$, which
is injective, it follows that $\{J(c)\}$ is clopen and $\{J(d)\}$ is not clopen in $\range(J)$ with respect to the ordinary Baire topology.
Now we define a continuous map $F:\Baire\to\Baire$ by
\[F(p):=c+|d-p|\]
for all $p\in\Baire$, where all arithmetic operations are meant pointwise.
It is clear that $F$ is continuous with $F^{-1}\{c\}=\{d\}$.
Let us now assume that some map $G:\In\Baire\to\Baire$ has the property
$FJ^{-1}=J^{-1}G$. In particular $\dom(G)=\range(J)$ in this situation. Then we obtain
\[G^{-1}\{J(c)\}=(J^{-1}G)^{-1}\{c\}=(FJ^{-1})^{-1}\{c\}=\{J(d)\}.\]
That is, although the set $\{J(c)\}$ is clopen in $\range(J)=\dom(J^{-1})$, its preimage
under $G$ is not clopen in $\dom(G)=\range(J)$
and hence $G$ cannot be continuous.
\end{proof}

\section{Conclusions}
\label{sec:conclusions}

\begin{figure}[htbp]
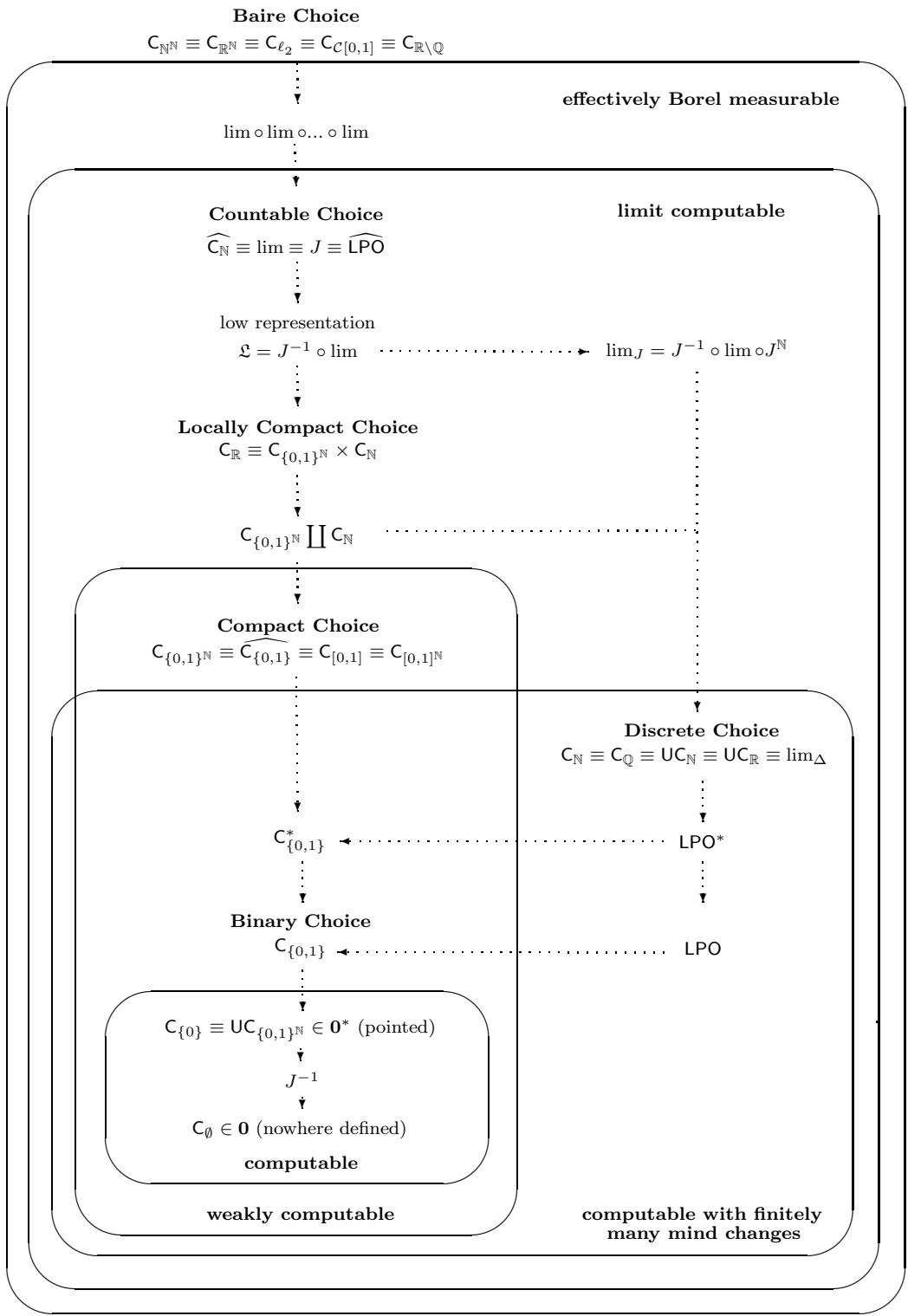

\begin{center}
\begin{scriptsize}
\input choice.pic
\end{scriptsize}
\caption{Closed choice in the Weihrauch lattice}
\label{fig:choice}
\end{center}
\end{figure}

We summarize some of the results that we have obtained in tables and figures.
Figure~\ref{fig:choice} extends the results provided in \cite[Figure~6]{BG09b}.
Here $\botW$ denotes the Weihrauch degree of the nowhere defined functions
and one obtains as $\botW^*$ the degree of all pointed computable multi-valued
functions on represented spaces.
The table below gives a list of some classes of multi-valued functions on represented
spaces that can be characterized by choice for certain spaces.
The given topological counterparts are at least correct for computable
Polish spaces and in some cases they have only been proved for single-valued functions.

\begin{table}[htb]
\label{tab:choice}
\begin{center}
\begin{footnotesize}
\begin{tabular}{ll}
{\bf Choice} & {\bf Class of functions (topologically)}\\\hline
$\C_{\{0\}}$ & computable (continuous) \\
$\C_\IN$  & computable with finitely many mind changes (piecewise continuous) \\
$\C_\Cantor$ & weakly computable (upper semi-continuous compact-valued selectors) \\
$\C_{\IN\times\Cantor} $ & weakly computable with finitely many mind changes \\
$\widehat{\C_\IN}$ & limit computable ($\Sigma^0_2$--measurable)\\
$\C_\Baire$ & effectively Borel measurable (Borel measurable)
\end{tabular}
\end{footnotesize}
\end{center}
\end{table}

The notion ``weakly computable with finitely many mind changes'' has not been used
before and is an ad hoc creation just for the purposes of this table.


\section*{Acknowledgement}

We thank the anonymous referees for several valuable corrections, comments and remarks that
helped to improve the final version of this paper.


\bibliographystyle{elsart-num-sort}
\bibliography{../../bibliography/new/lit,local}


\end{document}